\documentclass[11pt,reqno,twoside]{amsart}
\usepackage{hyperref}
\usepackage{amsfonts,amssymb,amsmath,amsthm}
\usepackage{graphicx}
\usepackage{caption}
\usepackage{enumerate}
\usepackage{bm}
\usepackage{float} 

\usepackage{pgfplotstable} 
 \pgfplotsset{compat=1.15}
 \usetikzlibrary{arrows}
\usetikzlibrary{positioning}
\usepackage[T1]{fontenc}
\usepackage{upquote} 

\usepackage{geometry}
\newtheorem{theorem}{Theorem}[section]

\newtheorem{corollary}[theorem]{Corollary}
\newtheorem{definition}[theorem]{Def{}inition}
\newtheorem{lemma}[theorem]{Lemma}

\newtheorem{proposition}[theorem]{Proposition}

\newtheorem{remark}[theorem]{Remark}
\numberwithin{equation}{section}

\newcommand{\ds}{\displaystyle}

\newcommand{\R}{{\mathbb R}}
\newcommand{\C}{\mathbb C}

\newcommand{\Spo}{\S_{\Gamma}} 
\newcommand{\N}{\mathcal{N}}              
\renewcommand{\R}{\mathbb{R}\,}           
\newcommand{\G}{\bm G\,}                  
\renewcommand{\div}{\mathrm{div}\,}       
\newcommand{\curl}{{\bf{curl}}\,}         
\newcommand{\divS}{\mathrm{div}_\Gamma}   
\newcommand{\curlS}{\mathrm{curl}_\Gamma} 

\newcommand{\e}{{\bm e}}               
\newcommand{\f}{{\bm f}}               
\newcommand{\g}{\bm g}                 
\newcommand{\J}{\mathcal J}            
\newcommand{\jc}{\bm j_{\bm c}}        
\newcommand{\n}{\bm n}                 
\newcommand{\q}{\bm q}                 
\renewcommand{\t}{\bm t}               
\renewcommand{\u}{\bm u}               
\renewcommand{\v}{\bm v}               
\newcommand{\w}{\bm w}                 
\newcommand{\x}{\bm x}                 
\newcommand{\y}{\bm y}                 
\newcommand{\z}{\bm z}                 
\newcommand{\barz}{\overline{\bm z}}   
\newcommand{\baru}{\overline{\bm u}}   
\newcommand{\barv}{\overline{\bm v}}   
\newcommand{\barw}{\overline{\bm w}}   
\newcommand{\tth}{{\tilde h}}                
\newcommand{\ttz}{{\widetilde{\bm z}}}    

\newcommand{\Hcurl}{{H}({\bf curl};\Omega)}            
\newcommand{\Hzcurl}{{H}_0(\mathrm{\bf curl};\Omega)}  

\newcommand{\Hdiv}{{H}(\mathrm{div};\Omega)}     
\newcommand{\Hzdiv}{{H_0}(\mathrm{div};\Omega)}  

\newcommand{\Lt }{L_t^2(\Gamma)}                
\renewcommand{\H}{{\bf H} }                     
\renewcommand{\L}{{\bf L}^2(\Omega) }           
\renewcommand{\S}{\mathcal{S}}                  
\newcommand{\U}{\mathcal{U}}                    
\newcommand{\Th}{\mathcal{T}_{h}}               
\newcommand{\HdivS}{{H}(\mathrm{div}_\Gamma;\Gamma)}    
\newcommand{\HcurlS}{{H}(\mathrm{curl}_\Gamma;\Gamma)}  

\newcommand{\Eh}{{\mathcal{E}_\Gamma^h}}        
\newcommand{\Lh}{{\mathcal{L}_h}}               
\newcommand{\Nkh}{{\mathcal{N}^h_k}}            
\newcommand{\EGh}{\mathcal{E}_{\Gamma}^h}       

\def\i{\mathrm{i}}

\begin{document}

\title{Boundary Control of Time-Harmonic Eddy Current Equations}

\thanks{
This work is partially supported by NSF grants DMS-2110263, DMS-1913004, DMS-2111315, the Air Force Office of Scientific Research (AFOSR) under Award NO: FA9550-19-1-0036 and FA9550-22-1-0248, and Department of Navy, Naval PostGraduate School under Award NO: N00244-20-1-0005.
}

	\author{Harbir Antil}
	\address{Harbir Antil. Department of Mathematical Sciences and the Center for Mathematics and Artificial Intelligence (CMAI), George Mason University, Fairfax, VA 22030, USA.}
	\email{hantil@gmu.edu}
	\author{Hugo D\'iaz}
	\address{Hugo D\'iaz. Department of Mathematical Sciences, University of Delaware, Newark, DE 19716, USA.}
	\email{hugodiaz@udel.edu}

\begin{abstract}
Motivated by various applications, this article develops the notion of boundary control for Maxwell's equations in the 
frequency domain. Surface curl is shown to be the appropriate regularization in order for the optimal control problem
to be well-posed. Since, all underlying variables are assumed to be complex valued,  the standard results 
on differentiability do not directly apply. Instead, we extend the notion of Wirtinger derivatives to complexified Hilbert spaces.
Optimality conditions are rigorously derived and higher order boundary regularity of the adjoint variable is established.
The state and adjoint variables are discretized using higher order N\'ed\'elec finite elements. The finite element space
for controls is identified, as a space, which preserves the structure of the control regularization. Convergence of the fully discrete 
scheme is established. The theory is validated by numerical experiments, in some cases, motivated by realistic 
applications.
\end{abstract}



\maketitle

\section{Introduction}

\subsection{Motivation}
\label{s:mot}

In recent years, problems related to controllability and optimal control constrained by Maxwell's equations have received a significant amount of attention in both time and frequency domain, see for instance a series of works \cite{MR1750615,MR3460108,caselli,MR3023745,MR984833,MR1824104,MR1932973,MR3259029,MR3515109,MR2891467,MR2957021,MR2925786,MR3105785,MR3377426}. Most of the existing literature focuses on the case where the control is in the interior, though the boundary control case can be found in some limited number of references
\cite{MR1750615,MR984833,MR1824104,MR1932973}. 
The novelty in this paper is to work in the frequency domain; where the previous works are in the time domain. The 
paper also deals with appropriate tangential traces (control space) of $H(\curl)$ when the  physical domain  is just Lipschitz polyhedral, unlike the previous works where either no numerical method is given or the domain is assumed to be smooth. This makes the approach introduced in these articles somewhat impractical. Indeed, no numerical examples are provided.
 We also present full convergence analysis for the fully discrete scheme, and a numerical implementation with finite elements. In addition, the paper works under a generic setting where all the variables are complex valued and it introduces a framework to carry out complex differentiation in Hilbert spaces.

Maxwell's equations in the non-smooth setting (Lipschitz polyhedral domains) with \linebreak non-homogeneous boundary conditions is rather recent \cite{buffa:I,buffa:II,tartar}. In addition, the underlying functional analytic framework is delicate. Nevertheless, such problems with non-homogenous boundary conditions, with non-smooth boundaries, do occur in realistic applications; for instance, microwave oven, also see \cite{bossavit}.
It is critical to resolve the physical domain, for instance, using the finite element method, which allows a 
simple implementation and it has a  well-studied theoretical framework, see \cite{MR2059447}.

Motivated by \cite{bossavit}, the articles \cite{BeRoSa2005-2,BeRoSa2005} introduced yet another boundary value problem for Maxwell equation (cf.~\eqref{eq:Rodolfo}) where the boundary conditions are on certain electrodes. Notice that the non-homogeneous boundary conditions of the type considered in the current paper also appear in scattering theory of electromagnetic fields. For instance, the incident and scattered  fields denoted by  $\bm E^i$ and $\bm E^s$, respectively, satisfy the ``boundary'' condition: $\bm E^s\!\times\! \n=-\bm E^i\!\times\! \n$. Here $\n$ is the outward unit normal.
These works forms the motivation for us to study boundary optimal control of time-harmonic Maxwell's equations.

\subsection{Problem Setup}

Let $\Omega \subset \R^3$ be a polyhedron with a Lipschitz continuous boundary denoted by $\Gamma$. Moreover, let the current density $\jc \in L^2(\Omega;\mathbb{C}^3)$, a desired field vector $\u_d$ in $ L^2(\Omega;\mathbb{C}^3)$, and symmetric and positive definite functions ${\bm \kappa}$ and ${\bm \mu}$  in  $L^\infty(\Omega;\R^{3\times 3})$,  
positive  constants  $\alpha$ and $\beta$,  and $\omega\neq 0$, be given. If $(\u,\z)$ represents the state-control pair, then 
the goal of this work is to study the following boundary optimal control problem:
\begin{subequations}\label{eq:ocp}
	\begin{align} \label{eq:obFun}
            \displaystyle \min_{(\u,\z) \in U \times Z} \left\{ \J(\u,{\bm z}) := f(\u) 
            	+\frac12 \left( \alpha \| \curlS \z\|_{L^2(\Gamma)}^2 
            	+ \beta \| \z \|^2_{L^2(\Gamma)} \right) \right\}, 
	\end{align}
subject to the time-harmonic Maxwell's equations as constraints 
	\begin{align}\label{eq:state}
	\begin{aligned}
		\curl \left(  {\bm \mu}^{-1}\curl \u\right)+(\i\omega){\bm \kappa} \u &= \jc \quad \mbox{in } \Omega,  \\
                    \u \times \n &= \z\times \n		\quad \mbox{on }  \Gamma . 
	\end{aligned}	                    
	\end{align}
\end{subequations}	
In \eqref{eq:obFun}, $\curlS$ denotes the scalar surface curl. For a precise definition of $\curlS$, the surface divergence $\divS$, tangential gradient $\nabla_\Gamma$, and the tangential vector curl $\bm \curlS$,  when $\Omega$ is a Lipschitz polyhedron, see \cite{buffa:I,buffa:II}.

To the best of our knowledge, this is the first work on boundary control of Maxwell's equations with complete  analysis and numerical implementation in the time harmonic setting.
Another key aspect that sets us apart from the existing literature on optimal control of Maxwell type 
equations is the fact we are working in the complex setting, without assuming splitting between real and 
imaginary parts.
 
However, this introduces additional challenges. For instance, even the standard 
quadratic cost functional 
	\begin{align}
	 f(\u) := \frac12 \int_\Omega |\u-\u_d|^2 \,d\x , \label{def:quadraticTerm}
	\end{align}
is not differentiable in the (complex) G\^ateaux sense, thus making most of the existing gradient
based optimization algorithms not directly applicable.  Provided that we can define an appropriate
notion of derivatives, the complex structure leads to elegant analysis. Additionally, 
most of the modern programming languages can handle complex arithmetic, therefore it is  
meaningful to work under the complex field directly.
In the particular case of quadratic functionals of the form (\ref{def:quadraticTerm}), the differentiability problem can be addressed via
{\it directional derivatives} as in  \cite[Sec. 3.2]{MR3515109}, where they consider an optimal current problem with a complex control related to impressed currents (internal control), see also \cite[Sec. 4.1]{caselli}. Our approach allows us to consider a much wider family of functionals, and leads to the same derivative for the quadratic case.

Our notion of derivatives is motivated by the {\it Wirtinger derivative} (1927) \cite{Wirtinger27}, which
is a well-known concept in finite dimensions. Its origin  can be traced back at least to 1899 in a work 
of {\it J.H. Poincar\'e} on potentials \cite[Th\'eor\`eme 8]{Poincare1898}. This notion of derivatives is 
motivated by splitting of a function into  its real and imaginary components. However, it allows one to
work in the complex regime, without using any such splitting. 
 Wirtinger derivatives have been used in the finite dimensional optimization problems at least since 
the 60's in the works of Levinson,  Mond, Hanson and Kaul, cf.  \cite{Levinson66,MondHanson68,Kaul71}, 
and  in the engineering community since the 80's, see \cite{Brandwood1983}, with applications from 
{\it signal theory}  to {\it Machine learning}, see \cite{Bouboulis2012}. We refer to  \cite{Kaup83,kreutz}  
for more details on Wirtinger derivatives. 

In this work, we extend the notion of derivatives from finite dimensions to infinite dimensions with 
complex fields and rigorously derive the optimality conditions at the continuous level and identify 
continuous gradients. We discretize our state and adjoint variables using higher order N\'ed\'elec finite 
elements and for the control we use the lowest order N\'ed\'elec elements on each boundary face, 
which by construction have continuous tangential components. 
Next, we establish convergence of our numerical scheme. Numerical experiments confirms
 our theoretical findings. In particular, in our first experiment, motivated by
\cite{BeRoSa2005-2,BeRoSa2005}, we consider a realistic application with non-homogeneous
boundary conditions where we first derive an explicit solution and next we validate our 
N\'ed\'elec finite element implementation against this explicit solution, the expected rate of 
convergence is observed. In the second experiment, we study the convergence of optimal
control problem, see section~\ref{s:num} for more details.

\noindent 
{\bf Outline:} In section~\ref{s:not}, we introduce some notation and establish the well-posedness of the state equation. 
Section~\ref{sec:rcf} first establishes existence of solution to the control problem. Next, section~\ref{s:Wirderiv} introduces
the notion of Wirtinger derivatives on complex Hilbert spaces and derives abstract optimality conditions. 
This is followed by a rigorous derivation of optimality conditions for our problem in section~\ref{s:optcond}. Well-posedness
of adjoint equation has also been established. Additional regularity for the adjoint equation and the optimality system   
have been provided in section~\ref{s:addreg}. Section~\ref{s:discprob} introduces the discrete optimal control problem.
Details on imposing non-zero boundary conditions have been provided in section~\ref{s:discbc}. Moreover, 
section~\ref{s:discsoln} is devoted to best approximation results for the state equations. The precise choice of 
the control space is discussed in section~\ref{s:ctrlspace}. Section~\ref{s:discopt} discusses the regularity of discrete 
adjoint equation and derives the discrete optimality system. Finally, in section~\ref{s:convan}, we provide convergence 
analysis of the optimal control problem. In Subsection~\ref{s:reg}, we establish that the lower order terms can be dropped,
i.e., $\beta$ in \eqref{eq:obFun} can be set to zero.
Section~\ref{s:num} is devoted to numerical examples which confirms the theoretical results.

\section{Notation and preliminary results}
\label{s:not}
From now on, if $X$ is a set of scalars, we use the  notation ${\bf X} $ to denote $(X)^3$, i.e., vectors.  For a vector space $V$, we will use the notation $V^*$ at times to refer to the space of linear functionals on $V$, and at other times we refer to the space conjugate linear functionals on $V$ (see, for example, \cite[p. 168]{OdDe2018}), this choice will be clear from the context. In what follows, we will need to identify the restriction of functions onto polygonal faces of $\Omega$. For this purpose we use the notation $\Gamma = \bigcup \overline{\Gamma}_i$, where each $\Gamma_i$ is an open subset of $\Gamma$ such that its closure is a face of $\Gamma$ and the $\Gamma_i$ are pairwise disjoint.  We use $\n$ to denote the outward unit normal. Next, we define the Hilbert spaces, endowed  with their standard inner products and norms: 
\begin{equation} \label{def:Spaces}
  \begin{aligned}
L^2(\Omega):=& L^2(\Omega;\mathbb{C}),\\
\Hzdiv :=& \{\v\in \Hdiv \colon\,  \gamma_{\n}(\v) := \v \cdot \n=0 \mbox{ on } \Gamma \},\\
\Hcurl:=& \left\{ \v\in {\bf L}^2\left(\Omega\right)\colon\, \curl \v \in {\bf L}^2\left(\Omega\right)\right\}, 
\\
 \Lt:=& \bigl\{\v \in L^2\left(\Gamma;\C^3\right)\colon\, \v\cdot \n= 0 \mbox{ a.e. on } \Gamma\bigr\}, \\
\HcurlS :=& \left\{\v\in \Lt\colon\, \curlS \v\in L^2(\Gamma) \right\},\\
\HdivS :=& \left\{\v\in \Lt\colon\, \divS \v\in L^2(\Gamma) \right\},\\
H_{\!\bm -}^{{\frac 12}}(\Gamma):=&\left\{\bm \lambda\in \Lt\colon\, \bm \lambda|_{\Gamma_i}\in \H^{{\frac 12}}(\Gamma_i), \mbox{ for each face $\Gamma_i \subset \Gamma$}    \right\},\\
  \end{aligned}
\end{equation}
where all the differential operators are in the sense of distributions, cf. \cite[Ch. I]{MR548867}. Unless stated otherwise, we will use the notation $(\cdot, \cdot)_{0,\Omega}$ and $\|\cdot\|_{0,\Omega}$ to denote the $L^2$-inner product and norm (respectively), regardless of whether the functions are scalar-valued, vector-valued, etc.  Similarly, we will use the notation $\langle \cdot, \cdot \rangle_\Gamma$ to denote the duality pairing of $H^{-{\frac 12}}(\Gamma)$ and $H^{{\frac 12}}(\Gamma)$. 

\noindent 
{\bf Control space.} 
We now define the set of {\it admissible controls} for our optimal control problem \eqref{eq:ocp}: 
\begin{align}
\label{eq:Z}
  Z:=&  \bigl\{\z\in \HcurlS:    \curlS \z\in L_0^2(\Gamma) \bigr\}, 
\end{align}
 endowed  with the norm on $\HcurlS,$ given by 
	\begin{equation}\label{eq:curlS}
		 \| \z \|_{\curlS} := \|\z\|_{{\bm L}^2(\Gamma)} + \| \curlS \z \|_{L^2(\Gamma)} . 
	\end{equation}
	
 We further emphasize that the norm definition in \eqref{eq:curlS} has motivated the control regularization in \eqref{eq:obFun}. In Subsection~\ref{s:reg}, we establish that the lower term in regularization can be dropped. Furthermore, in \eqref{eq:Z}, $L^2_0(\Gamma)$ is the space of $L^2$-functions with zero mean. The zero mean condition is a natural restriction for the problem,  related to the identity $\div \curl \u=0$ for $\u\in \Hcurl$; see for instance, \cite[Corollary 5.4]{BCS2002}. This condition also appears naturally in the finite element method setting, cf. \cite[Sec. 2]{MR3522965}.   Whenever we write $a \lesssim b$ in what follows, we mean that $ a \leq Cb$, where $C$ is a positive non-essential constant and its value might change at each occurrence. 
 
 \subsection{Tangential traces and Green's identities for $\Hcurl$}
 \label{ss:state}
We begin this section by defining the {\it tangential traces} of functions in $\mathbf H^1(\Omega)$, since from the boundary condition in  (\ref{eq:state}),  it is clear that these are the traces that are being imposed by the control.  The material in this subsection is known and we refer the interested reader to \cite{buffa:I} and \cite[Chapter 16]{MR3930592} for more details.  Recall that we are considering $\Omega$ to be a polyhedron, therefore the outer unit normal $\n$ is well-defined almost everywhere on $\Gamma,$  and  along each edge of $\Omega$ one of the tangential components is continuous.
\begin{definition}
\label{def:gtgT}
 The tangential traces of $\v$, defined from $\H^1(\Omega)$  onto $H_{\!\bm -}^{{\frac 12}}(\Gamma)$, are given by
    \[
    	\gamma_t \v :=\gamma\v \times \n,\qquad \text{and} \qquad 
    		\gamma_T \v :=\n\times \left(\gamma\v \times \n\right).
    \]
   where $\gamma$ denotes the standard restriction of $\v$ on $\Gamma$ in the trace sense.
\end{definition} 
 We now define the Hilbert spaces $ H_{\perp}^{\frac 12}(\Gamma)$ and $ H_{||}^{\frac 12}(\Gamma)$, see \cite[Prop. 2.6]{buffa:I}, as the image of the maps $\gamma_t$ and $\gamma_T$ restricted to $\H^1(\Omega)$. Moreover, we have the following result 
\begin{lemma} 
\label{lem:TrH1}
The following maps are linear, continuous and surjective
\begin{align*}
 \gamma_t:\H^1(\Omega)\mapsto H_{\perp}^{\frac 12}(\Gamma), ~ \mbox{ and }
 \gamma_T:\H^1(\Omega)\mapsto H_{||}^{\frac 12}(\Gamma).
\end{align*}
\end{lemma}
\begin{proof}
 See \cite[Proposition 2.7]{buffa:I}.
\end{proof}

 \begin{definition} 
 The spaces 
    \begin{align}
     (H_{\perp}^{-\frac 12}(\Gamma),\| \cdot\|_{\perp,-\frac 12,\Gamma}),~ \mbox{ and }~
     (H_{||}^{-\frac 12}(\Gamma), \| \cdot\|_{||,-\frac 12,\Gamma})
    \end{align}
   are the dual spaces of $ H_{\perp}^{\frac 12}(\Gamma)$ and $ H_{||}^{\frac 12}(\Gamma)$ (endowed with dual norms), respectively. 
    In this case, $\Lt$ is taken as the pivot space.
 \end{definition}
 
However, Lemma~\ref{lem:TrH1} is not directly applicable in our setting. Indeed, the correct function space for \eqref{eq:state} is  $\Hcurl$ and 
 not $\H^1(\Omega)$. Notice that $\Hcurl$ is less regular than $\H^1(\Omega)$, but the dual space  of its trace space is more delicate than $\bm H^{-\frac{1}{2}}(\Gamma).$  

To define weaker versions of $\gamma_t$ and $\gamma_T$, we first note that for $\v \in \boldsymbol{\mathcal C}^\infty(\overline{\Omega})$, the image of the maps $\gamma_t \v$ and $\gamma_T \v$ belong to $H_{\!\bm -}^{{\frac 12}}(\Gamma)$, according to Definition~\ref{def:gtgT}. Since no restrictions are imposed on the normal component of $\gamma\u$,  those maps can be extended by density to elements of $\Hcurl$.

In order to obtain a Green's identity where both functions are in $\Hcurl$, we need to properly define the ranges of $\gamma_t$ and $\gamma_T$ acting on $\Hcurl$.  We again refer to \cite{buffa:I} and \cite[Ch. 16]{MR3930592} as well as \cite[p. 58]{MR2059447} for more details.  Following the notation of \cite{buffa:I}, we define 
\begin{align*}
	H_{||}^{-{\frac 12}}(\divS;\Gamma) :=&~ 
	\left\{ \bm \lambda \in H_{||}^{-{\frac 12}}(\Gamma):\, \divS \bm \lambda \in H^{-\frac 12}(\Gamma) \right\},\\
	H_{\perp}^{-{\frac 12}}(\curlS;\Gamma) :=&~ 
	\left\{ \bm \lambda \in H_{\perp}^{-{\frac 12}}(\Gamma):\, \curlS \bm \lambda \in H^{-\frac 12}(\Gamma) \right\},
\end{align*}
endowed with the norms 
\begin{align}
\begin{aligned}
     \| \bm \lambda\|_{H_{||}^{-{\frac 12}}(\divS;\Gamma)}:=&~ 
      \| \bm \lambda\|_{||,-\frac 12,\Gamma}+ \|\divS \bm \lambda\|_{-\frac 12,\Gamma},\\
     \| \bm \lambda\|_{H_{\perp}^{-{\frac 12}}(\curlS;\Gamma)}:=&~ 
      \| \bm \lambda\|_{\perp,-\frac 12,\Gamma}+ \|\curlS \bm \lambda\|_{-\frac 12,\Gamma}.
\end{aligned}
     \label{def:normTraceHcurl}
\end{align}
Moreover, we have  $H_\perp^{-{\frac 12}}(\curlS;\Gamma) := (H_{||}^{-{\frac 12}}(\mathrm{div}_\Gamma;\Gamma))^*$, where $\Lt$ is used as  pivot, and 
the following holds:
\begin{lemma}\label{lemma:surjTanTrace}
 The maps,
 \begin{align*}
 \begin{aligned}
  \gamma_t: \Hcurl\mapsto H_{||}^{-{\frac 12}}(\divS;\Gamma),  \mbox{ and }
  \gamma_T: \Hcurl\mapsto H_{\perp}^{-{\frac 12}}(\curlS;\Gamma)
 \end{aligned}
 \end{align*}
are linear, continuous and surjective. 
\end{lemma}
\begin{proof}
 See \cite[Thm. 5.4]{buffa:II}.
\end{proof}
With these definitions we have the following Green's identity:
\begin{theorem}[{\cite[Thm. 3.31]{MR2059447}}]
  The space $H_{||}^{-{\frac 12}}(\mathrm{div}_\Gamma;\Gamma)$ is a Hilbert space. The continuous linear mappings $\gamma_t$: $\Hcurl \to H_{||}^{-{\frac 12}}(\mathrm{div}_\Gamma;\Gamma)$  and $\gamma_T$ :  $\Hcurl \to H_\perp^{-{\frac 12}}(\curlS;\Gamma)$ are surjective, and for all $\v, \bm \phi \in \Hcurl$ 
\begin{align}
(\v,\nabla \times \bm \phi)_{0,\Omega} - (\nabla \times \v, \bm \phi)_{0,\Omega} = \langle \gamma_t \v, \gamma_T\bm \phi\rangle_{\Gamma^*},  \label{eqn:GreenCurl}
\end{align}
where $\langle \cdot, \cdot \rangle_{\Gamma^*}$ is the duality pairing between $H_{||}^{-{\frac 12}}(\mathrm{div}_\Gamma;\Gamma)$ and $H_\perp^{-{\frac 12}}(\curlS;\Gamma)$. 
\end{theorem} 
Now, given $\g \in H_{||}^{-{\frac 12}}(\mathrm{div}_\Gamma;\Gamma)$ we define 
\begin{align*}
H_{\bm g}(\curl;\Omega):=&  \left\{\v \in \Hcurl\colon\, \gamma_t\v=\g \mbox{ on } \Gamma \right\}.
\end{align*}
%

\subsection{Well-posedness of the state equation}

First, let us introduce the {\it sesquilinear form}\newline  ${a:\Hcurl\times \Hcurl \to \C}$ given by 
\begin{equation}
 a(\u,\v):= \int_\Omega {\bm \mu}^{-1} \curl \u \cdot \curl{\barv}\,d\x+(\i\omega)\int_\Omega {\bm \kappa} \u \cdot \barv\, d\x. \label{def:a}
\end{equation}
The following lemmas show the well-posedness of the state equation (\ref{eq:state}).

\begin{lemma}[] \label{lemma:bijection_a}
Given $\f\in \Hcurl^*$, the problem, find $\u\in \Hcurl$ such that 
\begin{align}
     a(\u,\v)= \langle\f,\v \rangle\qquad \forall \v\in\Hcurl,\label{def:AbstStateEq}
\end{align}
 is a well-posed problem  in the sense of {\it Hadamard}, where $\langle\cdot,\cdot \rangle$ denotes the duality paring between $\Hcurl$ and its dual. 
\end{lemma}    
\begin{proof}
     The hypothesis on ${\bm \mu}$ and ${\bm \kappa}$ guarantees that $a(\cdot,\cdot)$ defines a sesquilinear, bounded  and coercive form.  The Lax-Milgram lemma implies that (\ref{def:AbstStateEq})  has a unique solution; moreover, we have
     \begin{align}
    c({\bm \kappa}, {\bm \mu};\omega)\|\u \|^2_{\curl,\Omega}\leq  |a(\u,\u)|  \lesssim |\langle\f,\u\rangle |,
     \end{align}
where $c$ is a positive constant that depends on $\omega$ and the eigenvalues of ${\bm \mu}$ and ${\bm \kappa}$.
\end{proof}

\begin{lemma}[]\label{lemma:state_sol} Given $\z \in \HcurlS$ and $\bm j_c\in \L$, the problem, find $\u \in H_{\z\times \n}(\curl;\Omega)$ such that  
\begin{align}\label{eq:state_sol}
a(\u,\v) = \int_\Omega \jc \cdot \overline{\v}\,d\x  \qquad \forall \v \in \Hzcurl,
\end{align}
 is well-posed. 
\end{lemma}
\begin{proof}
Let $\z \in \HcurlS$ be given. Then $\z\times \n \in \gamma_t \Hcurl$, which follows from the much more general results on weak rotations,  see \cite[Prop. 16.16 and Eq. (16.42)]{MR3930592}. Therefore fom Lemma \ref{lemma:surjTanTrace}, there exists $\u_{\z}\in \Hcurl$ such that $\gamma_t \u_{\z}=\z\times \n$ and
\begin{align}
         \|\u_{\z} \|_{\curl,\Omega} \lesssim 
           \|\z\times \n\|_{{H^{-{\frac 12}}_{||}(\divS;\Gamma)}}\lesssim \|\z\|_{\curlS} , \label{eqn:contLifting}
\end{align}
where the last inequality follows because of the isometry 
$ \|\z\times \n\|_{{H^{-{\frac 12}}_{||}(\divS;\Gamma)}} = \|\z\|_{H_\perp^{-{\frac 12}}(\curlS;\Gamma)}$,
see \cite[p. 448]{MR3930592}, 
and 
$$
\|\z\|_{H_\perp^{-{\frac 12}}(\curlS;\Gamma)} 
\le C \|\z\|_{\curlS},
$$
which follows from the compact embedding of $L^2(\Gamma)$ into $H^{-\frac 12}(\Gamma)$. 

    Now, we look for $\u_0\in \Hzcurl$ such that 
    \begin{align}
        a(\u_0,\v)= \int_\Omega \jc \cdot \overline{\v}\,d\x - a(\u_{\z},\v) \qquad \forall \v \in \Hzcurl.
    \end{align}
    From Lemma \eqref{lemma:bijection_a}  and \eqref{eqn:contLifting}, it follows that  this problem is well-posed, and 
    $$
    \|\u_0 \|\lesssim \|\jc\|_{0,\Omega}+\|\z \|_{\curlS}.
    $$
    Finally, $\u :=\u_0+\u_{\z}$ is the unique solution of (\ref{eq:state_sol}) and 
    $$\|\u\|_{\curl,\Omega}\lesssim \|\jc\|_{0,\Omega}+ \|\z \|_{\curlS},$$
    which finishes the proof.
\end{proof} 

From the previous analysis, $\u$ depends on both $\z$ and $\jc$. The goal of the next section is to reduce the cost functional $\J(\u,\z) = \J(\u(\z;\jc),\z)$ to be only a function of $\z$, and then derive the optimality conditions.


\section{Reduced cost functional and its derivative }
\label{sec:rcf}


For the remainder of the paper, we will assume that  $\jc $ is given. By introducing the control-to-state map, we can obtain the so-called reduced optimization problem. The solution, or control-to-state,  map is affine and it is given by   
\begin{align}
\S: Z & \to \Hcurl \hookrightarrow \L  \nonumber\\
                       \z & \mapsto \S \z:=\u,\label{def:S} 
\end{align}
where  $\u$  is the unique solution to (\ref{eq:state_sol}) with right-hand-side $\jc$ and the boundary condition $\z\times \n$. The notation $\hookrightarrow$ indicate the continuous embedding, as a result, we can consider $\S : Z \rightarrow \L$. 
The solution operator $\S$ is an affine map, and it is common to split $\S$ into the part that depends on $\z$ and the part  that depends on $\jc$. We write
	\[
		\S\z = \Spo\z + \u_\Omega,
	\]
where $\Spo$ is the solution operator for the state equation with $\jc\equiv {\bf 0}$, and $\u_{\Omega}$ is the solution for the state equation when  $\z\equiv \bm 0$.  By Lemma~\ref{lemma:state_sol}, $\S$ is continuous and there exists a $C=C({\bm \mu},{\bm \kappa},\Omega)>0$ such that 
\begin{align}
    \|\S \z \|_{\curl,\Omega}= \|\u \|_{\curl,\Omega}\leq C\left(\|\jc \|_{0,\Omega} +\|\z \|_{\curlS}  \right).
    \label{eqn:contS}
\end{align}

Therefore, the {\it reduced cost functional} $j(\z):=\J(\S \z,\z)$ is also continuous, and we use the splitting above to write
\begin{align}\label{eq:reduced_cost}
            \displaystyle  \min_{\z\in Z} j(\z) =  
            \ds  \min_{\z\in Z} \frac{1}{2}\int_\Omega |\Spo \z-\widehat{\u}_d|^2\,d\x+\frac{\alpha}{2}\int_{\Gamma} |\curlS\z|^2\,dS
            +\frac{\beta}{2}\int_{\Gamma} |\z|^2\,dS, 
\end{align}
where $\widehat{\u}_d = \u_d-\u_\Omega.$ Therefore, without loss of generality, in what follows we will consider 
$\jc\equiv \bm 0$, then  $\S= \Spo$. Notice that in this case, $\u_\Omega= 0$. 

 Our goal now is to discuss the existence and uniqueness of solution to the reduced optimization problem
	\begin{align}
            \displaystyle  \min_{\z\in Z} j(\z)= \min_{\z\in Z} \J(\S \z ,\z).\label{def:ReducedFunctional}
	\end{align}
The proof follows immediately from the direct method of calculus of variations, we sketch it for completeness. 

\begin{theorem}[existence and uniqueness] The problem (\ref{def:ReducedFunctional}) has a unique solution $\barz\in Z$.
\end{theorem}
\begin{proof} 
Notice that $j(\cdot)$ is bounded below, therefore there exists an infimizing sequence $\{\z_n\}_{n=1}^\infty$ such that 
$\ds \inf_{\z \in Z} j(\z) = \lim_{n \rightarrow \infty}j(\z_n)$. The previous limit and the definition of $j(\cdot)$ implies that 
$\{\z_n\}_{n=1}^\infty$ is a bounded sequence in $Z$. Notice that $Z$ is closed subspace of a Hilbert space and
is therefore a Hilbert space itself. Thus the boundedness of sequence $\{\z_n\}_{n=1}^\infty$ implies that there exists 
a subsequence (not relabeled) that converges to $\widehat\z$ in $Z$. It then remains to show that $\widehat\z$ is the 
minimizer of \eqref{def:ReducedFunctional}. This immediately follows from weak lower semicontinuity of $j(\cdot)$. 
The uniqueness is a direct consequence of the strict convexity of $j(\cdot)$. 
\end{proof}

As usual, in order to f{}ind the optimal control $\barz$ we want to use a gradient-based method to identify the critical points of the cost functional $j(\cdot)$ in \eqref{eq:reduced_cost}. Nevertheless, there is a big difference between differentiability over typical real and complex fields. However, Fr\'echet and G\^{a}teaux differentiability on complex fields are quite similar, cf. \cite{Zorn1945a,Zorn1945b,Zorn1946}. Notice that, $|\Spo \z-\widehat{\u}_d|^2$ in \eqref{eq:reduced_cost} is smooth, but as we will see below, it is not complex G\^{a}teaux differentiable. Thus, to study \eqref{eq:reduced_cost}, we need some weaker notion for the  derivative of $j$.  
Next, we will develop a notion of derivative for $j$ that is more general than G\^{a}teaux or Fr\'echet derivatives for complex spaces but strong enough to have some notion of Taylor series expansion that will allow us to characterize critical points of $j$.  

Before we begin our discussion of derivatives, we introduce what it means for a complex-valued function to be $\mathcal C^1$, but not necessarily analytic, see for instance \cite[(2.1)]{Lempert1998}. 
\begin{definition}[$\mathcal{C}^1$ functions] 
\label{def:C1}
For a complex Banach space $\mathcal U$, suppose $\mathcal D\subset \mathcal U$ is open and $u:\mathcal D \to \C$ is a function.  If the directional derivatives
\begin{align}
 du(\z;\bm \xi)=\lim_{t \to 0} \frac{u(\z+t\bm \xi)-u(\z)}{t},\qquad t \in \mathbb R
\end{align}
exist for all $\z\in \mathcal D$, $ \bm \xi \in \mathcal U$, and $du:\mathcal D\times \mathcal U \to \C$ is continuous, we write $u\in \mathcal{C}^1(\mathcal D)$.
\end{definition}
Consider the function
\begin{equation}
 \label{eq:zsq}
 \z\mapsto \z\overline{\z}=|\z|^2, 
\end{equation}
which is $\mathcal{C}^1(\C)$ but not complex Fr\'echet differentiable, nor complex G\^{a}teaux differentiable. This shows that complex differentiability is too restrictive, especially in the context of optimization. A weaker notion of ``complex derivative'' which has most of the required properties in optimization is the so-called Wirtinger derivative, cf. \cite{Wirtinger27}.

In the next section, we extend the notion of Wirtinger derivatives to spaces which are the complexification of a real Hilbert space, and our analysis follows as in the case $f:\C \to \C,$ cf. \cite{van-den-Bos}.


\subsection{ Wirtinger derivatives on complexified Hilbert spaces} 
\label{s:Wirderiv}

Let $(\U,\|\cdot\|_\U)$ be the complexification of a real Hilbert space $(H, (\cdot ,\cdot)_H)$, and let $f:\mathcal{U}\to \C$ be a continuous function but not complex differentiable. We can extend $f$ to a function on $\mathcal{U}\times \mathcal{U}$. Let $g: \mathcal{U}\times \mathcal{U}\to \C$  be a continuous extension of $f$ such that 
\[
	g(\z, \barz) =f(\z)  \quad \forall \z\in \mathcal{U}.
\]
Now, let us assume $g$ is complex  Fr\'echet differentiable, and define 
\[
 \frac{\partial f}{\partial \z}\Bigr\lvert_{\z=\z_0} := \nabla g\left( \z_0, \barz_0\right)(\e_1),\qquad \qquad 
 \frac{\partial f}{\partial \barz}\Bigr\lvert_{\z=\z_0}  := \nabla g\left(\z_0, \barz_0 \right)(\e_2) , 
\]
where the unit vectors $\e_1$ and $\e_2$ will give us the first and second components of $\nabla g$. 
Even though, the existence of an extension $g$, which is complex Fr\'echet differentiable, of $f$ may look too restrictive, however such an extension exists when $f$ is continuous and $f$ is analytic for $\z$ and $\barz$; separately, cf.  \cite[p. 2]{Kaup83}, these conditions hold for the function in \eqref{eq:zsq} for instance. This follows from {\it Hartogs' theorem} or one of its generalizations to infinite-dimensional spaces; see for instance, \cite[Thm. 3.2]{Matos78}.  Now, for any $\z_0$ and $\delta \z$ in $\mathcal{U}$, the following limit exists and the resulting expression is linear in  $\delta \z$,
\begin{align}\label{def:dRf}
    \begin{aligned}
        \mathrm{d}^{\mathbb{R}} f(\z_0;\delta\z):=&  \lim_{\genfrac{}{}{0pt}{}{t\rightarrow 0}{\mathrm{Im}(t)=0}} \frac{ f(\z_0+t\delta \z)-f(\z_0)  }{t}\\
        =&\lim_{\genfrac{}{}{0pt}{}{t\rightarrow 0}{\mathrm{Im}(t)=0}} \frac{ 
        g\left(\begin{pmatrix} \z_0, \barz_0 \,\, \end{pmatrix}
        +t\begin{pmatrix} \delta\z, \overline{\delta\z} \,\, \end{pmatrix}\right)
        -g\begin{pmatrix} \z_0, \barz_0 \,\, \end{pmatrix}  }{t}\\
        =& \left(  \begin{pmatrix} \delta\z\,\,\\ \overline{\delta\z} \,\, \end{pmatrix}, \overline{\nabla g}(\z_0,\barz_0) \right)_{\mathcal{U}\times \mathcal{U}}\\
        =& \left(\delta \z, \overline{ \frac{\partial f}{\partial \z}}(\z_0)\right)_\U + \left(\overline{\delta \z},\overline{ \frac{\partial f}{\partial \barz}}(\z_0)\right)_\U.
    \end{aligned}
\end{align}

 \begin{theorem}
    Let $\U$ be the complexification of a real Hilbert space $H$,  and consider a continuous function 
    $f:\U\mapsto \mathbb{R}, \z \mapsto f(\z)$, such that $f$ is analytic in $\z$ and in $\barz$, separately, then $ \mathrm{d}^{\mathbb{R}} f(\z_0;\delta\z)$ defined in \eqref{def:dRf} exists. Moreover, 
    \begin{align}\label{eq:dRfmod}
\mathrm{d}^{\mathbb{R}} f(\z_0;\delta\z)  = &~  
   2 \mathrm{Re} \left( \frac{\partial f}{\partial \barz}(\z_0), \delta \z\right)_\U \qquad \forall \z_0,\delta \z \in \U.      
\end{align}
 \end{theorem}
\begin{proof}
 The existence of $\mathrm{d}^{\mathbb{R}} f(\z_0;\delta\z)$ follows from the previous analysis. 
 On the other hand, if $f$ is real-valued then by definition $\mathrm{d}^{\mathbb{R}} f(\z_0;\delta\z) $ is also real-valued, yielding 
 \begin{equation} \label{eq:partDerConj}
\frac{\partial f}{\partial \z}(\z_0)= \overline{ \frac{\partial f}{\partial \barz}}(\z_0).
\end{equation}
In order to prove this identity,  consider the inner product on $\mathcal{U}$ given by 
\begin{align*}
 \left ( \u_1+i\v_1,\u_2+i\v_2 \right)_{\mathcal{U}}:=
 \left ( \u_1,\u_2\right)_{H} +  \left ( \v_1,\v_2\right)_{H}
 +\i \left( \left ( \v_1,\u_2\right)_{H}-  \left ( \u_1,\v_2\right)_{H}  \right).
\end{align*}
Now, we consider the splitting between real and imaginary parts 
\begin{align*}
 \delta \z = \delta^{Re} \z + \i  \delta^{Im} \z,\quad 
 \frac{\partial f}{\partial \z}(\z_0) = f^{Re}_{\z} + \i f^{Im}_{\z}, \mbox{ and }
 \frac{\partial f}{\partial \barz}(\z_0) = f^{Re}_{\barz} + \i f^{Im}_{\barz}.
\end{align*}
In turn, from (\ref{def:dRf})
\begin{align*}
d^{\R}f(\z_0;\delta \z) =&
 \left( \delta \z, \overline{ \frac{\partial f}{\partial \z}}(\z_0)  \right)_{\mathcal{U}}
 +  \left ( \overline{\delta \z}, \overline{ \frac{\partial f}{\partial \barz}}(\z_0)  \right)_{\mathcal{U}}\\
 =&
 \left ( \delta^{Re} \z + \i  \delta^{Im} \z, f^{Re}_{\z} - \i f^{Im}_{\z}  \right)_{\mathcal{U}}
 +  \left ( \delta^{Re} \z - \i  \delta^{Im} \z, f^{Re}_{\barz} - \i f^{Im}_{\barz} \right)_{\mathcal{U}},
\end{align*}
and because  $\ds \mathrm{Im}\{ d^{\R}f(\z_0,\delta \z)\}=0$,  we have 
\begin{align*}
 \left( f^{Re}_{\z}, \delta^{Im} \z  \right)_{H}+ 
 \left (  \delta^{Re} \z, f^{Im}_{\z}   \right)_{H}  
 - \left ( f^{Re}_{\barz}, \delta^{Im} \z  \right)_{H}+
 \left (  \delta^{Re} \z, f^{Im}_{\barz}  \right)_{H} =0.
\end{align*}
In particular, if $\delta^{Re} \z=\bm 0$ we get 
\begin{align*}
 \left ( f^{Re}_{\z}, \delta^{Im} \z  \right)_{H}  
 - \left ( f^{Re}_{\barz}, \delta^{Im} \z  \right)_{H} 
 =0\Leftrightarrow  f^{Re}_{\z} = f^{Re}_{\barz}.
\end{align*}
In turn,  if $\delta^{Im} \z=\bm 0$ we get 
\begin{align*}
 \left (  \delta^{Re} \z, f^{Im}_{\z}   \right)_{H}  
 +\left (  \delta^{Re} \z, f^{Im}_{\barz}  \right)_{H} 
 =0\Leftrightarrow  f^{Im}_{\z} = -f^{Im}_{\barz}.
\end{align*}
Thus, 
 \begin{equation*} 
\frac{\partial f}{\partial \z}(\z_0)= \overline{\left( \frac{\partial f}{\partial \barz} (\z_0)\right)}, \mbox{ and therefore}
\overline{\left( \delta \z,  \overline{ \frac{\partial f}{\partial \z}}(\z_0)\right)}_\U=  
 \left( \overline{\delta \z},\overline{\frac{\partial f}{\partial \barz}}(\z_0) \right)_\U,
\end{equation*}
which concludes the proof.
\end{proof}

Our next goal is to relate $\mathrm d ^\mathbb{R} f$ given in \eqref{eq:dRfmod} with a gradient so that we can derive the f{}irst-order optimality conditions for problem \eqref{eq:obFun}.
In order to do that,  we identify $f$ with a real functional $u$, namely $u: H\times H\to \mathbb{R}$ that satisfies $f(\z)=f(\x+i\y)=u(\x,\y)$  for all $\z=\x+i\y $ in $ \mathcal{U}.$  
From the regularity of $g$ (as defined above) we obtain  
\begin{align}
\label{eq:secchar}
 \mathrm{d}^{\mathbb{R}} f(\z;\delta\z) = &
  \lim_{\genfrac{}{}{0pt}{}{t\rightarrow 0}{t\in \R}} \frac{u(\x+t\delta {\x}, \y+t\delta {\y})-u(\x,\y)}{t} \nonumber \\
  = & \left(\nabla u, \begin{pmatrix}
                        \delta {\x}\\ \delta {\y}
                       \end{pmatrix}
 \right)_{H\times H}  \\
 =& \left(\frac{\partial u}{\partial \x},  \delta {\x} \right)_H+ \left(\frac{\partial u}{\partial \y},  \delta {\y} \right)_H, \nonumber
\end{align}
where $\delta \z= \delta {\x}+i\delta {\y}$. Consequently, $f\in \mathcal{C}^1(\mathcal{U})$ 
(see Definition~\ref{def:C1}) and we have the following identities
\begin{align}
 u\left( \begin{pmatrix}\x, \y \end{pmatrix}+ \begin{pmatrix}\delta\x, \delta  \y \end{pmatrix} \right) =& ~ u \begin{pmatrix}\x, \y \end{pmatrix}+ \left(\nabla u(\widehat{\x},\widehat{\y}), \begin{pmatrix}
                        \delta {\x}, \delta {\y}
                       \end{pmatrix}\right),\label{def:RTaylor}\\
 f(\z+\delta\z)=&~ f(\z)+ \mathrm{d}^{\mathbb{R}} f(\widehat{\z};\delta\z), \label{def:CTaylor} 
\end{align}
for some $\widehat{\z}= \widehat{\x}+i\widehat{\y} $ in the  segment $[\z, \z+\delta \z].$ From now on, the expression $\ds \mathrm{d}^{\mathbb{R}} f(\z;\delta\z) $ will be called the \emph{$\R\!-$linear derivative} of $f$ at $\z$ in the direction $\delta \z.$ Many properties for $f$ can be obtained from  properties of extension $g$ or using the relation with $u$; for instance,  $f$ is (real) convex if and only if $u$ is convex. Also, one can determine the directions of steepest descent and stationary points. In fact, we have the following result.
\begin{theorem}[Steepest descent]
Let $f$, $g$ and $u$ be as above, then the direction of steepest descent for $f$ at $\z_0=\x_0+\i\y_0$ is given by
 \begin{align*} 
  \delta \z =&  -  \overline{\frac{\partial f}{\partial \z}}(\z_0)
  					= - \frac{\partial f}{\partial \barz}(\z_0),
\intertext{or equivalently}
  \begin{pmatrix}
   \delta \x\\\delta \y
  \end{pmatrix}
  =& - \nabla u (\x_0,\y_0).
\end{align*} 
\end{theorem}
\begin{proof}
 The result follows from the two characterizations of $\mathrm{d}^{\mathbb{R}} f$ given in \eqref{eq:dRfmod} and \eqref{eq:secchar}, 
 respectively. In the first case, we have also used \eqref{eq:partDerConj}. The proof is complete.
\end{proof}
From  the previous theorem we obtain 

$$ \ds\frac{\partial f}{\partial \z}(\z) = \frac{\partial u}{\partial \x}(\x,\y)+\i\frac{\partial u}{\partial \y}(\x,\y), ~\mbox{and }~   \ds\frac{\partial f}{\partial \barz}(\z) = \frac{\partial u}{\partial \x}(\x,\y)-\i\frac{\partial u}{\partial \y}(\x,\y).$$
It is common  to write these identities, called the {\it Wirtinger derivatives} (in finite dimensions) for a real-valued function $f$, as 
\begin{align}
\frac{\partial f}{\partial \z}&= \frac{\partial f}{\partial \x}+\i\frac{\partial f}{\partial \y},\label{eq:dzf} \\
 \frac{\partial f}{\partial \barz}&= \frac{\partial f}{\partial \x}-\i\frac{\partial f}{\partial \y}. \label{eq:dz*f} 
\end{align}
\begin{lemma}
Let $f$  be as above. Then, $\z_0=\x_0+i\y_0$ is a stationary point of $f$ if and only if
\[
 \frac{\partial f}{\partial \z}(\z_0)=\bm 0  \qquad \text{or} \qquad 
 		\frac{\partial f}{\partial \barz}(\z_0) =\bm 0.
\]
\begin{proof}
 It is enough to identify $f$ with $u$, then $\ds \nabla u (\x_0,\y_0)=\bm 0$, and using (\ref{eq:dzf}) and  (\ref{eq:dz*f}) the proof is complete.
\end{proof}
\end{lemma}
Finally, we give the following useful theorem holds. By $\mathrm{Dom}(f)$, we indicate the domain of function $f$. 
\begin{theorem} \label{thm:optimality}
 Let $f$ be as above, and assume $f:\mathrm{Dom}(f)\to \C$ with  $\mathrm{Dom}(f)$ convex. If $\z_0$ is an optimal point then
\begin{align}
	\mathrm{d}^{\mathbb{R}} f(\z_0;\z-\z_0)\geq 0,  \quad  \forall \z \in \mathrm{Dom}(f).\label{eqn:1stOrder}
\end{align}
In addition, if $f$ is (real) convex then \eqref{eqn:1stOrder} is a sufficient condition. 
\end{theorem}
\begin{proof}
 This result follows directly by identifying $f$ with $u$ and applying the well-known result for convex real-valued functions.
\end{proof}
%


\subsection{The $\mathbb{R}-$linear derivative of the reduced cost functional and optimality conditions}
\label{s:optcond}


Since the map $\z\mapsto \barz\cdot \z$ is not (complex) G\^{a}teaux differentiable, the same can be concluded for the 
reduced functional $j$ defined in (\ref{eq:reduced_cost}). The following lemma shows, however, that   
$\mathrm{d}^{\mathbb{R}} j$ is well-defined.

\begin{lemma} 
\label{lem:Gd}
Let $(X,\|\cdot \|_{X})$ be a complex Banach space, $(\mathcal{U}, (\cdot,\cdot)_\U)$ be a complex Hilbert space, $\widetilde\u\in \U$,  and  $\widetilde\S\in \mathcal{L}(X,\U)$. Then, the convex function $f(\z):= \|\widetilde\S\z -\widetilde\u \|_\U^2$  
has an $\R$-linear derivative given by 
\begin{align*}
    \mathrm{d}^{\mathbb{R}}  f(\z;\v)=&~ 2 \, \mathrm{Re} (\widetilde\S\z -\widetilde\u,\widetilde\S\v)_{\U},
\end{align*}
where $\|\u\|^2_\U:=(\u,\u)_\U$ and  $\mathrm{Re}(a)$ denotes the real part of $a\in \C$. 
\end{lemma}
\begin{proof}
We examine the difference quotient from the definition of $\mathrm d^\mathbb{R},$ cf. \ref{def:dRf}, namely: 
\begin{align}\label{eq:fz}
    \frac{f(\z+t\v)-f(\z)}{t}=&~ \frac{\|\widetilde\S(\z+t\v) -\widetilde\u \|_\U^2-\|\widetilde\S\z -\widetilde\u \|_\U^2}{t} \nonumber \\
                             =&~ \frac{\bigl( \|\widetilde\S\z  -\widetilde\u \|_\U^2
                             +2\, \mathrm{Re} (\widetilde\S\z -\widetilde\u,t\widetilde\S\v)_\U+ 
                             \|t\widetilde\S\v\|^2_X \bigr)-\|\widetilde\S\z -\widetilde\u \|_\U^2}{t} \nonumber  \\
                              =&~ \frac{  2\, \mathrm{Re} \,\bar{t}(\widetilde\S\z -\widetilde\u,\widetilde\S\v)_\U+ 
                              |t|^2 \|\widetilde\S\v\|^2_\U }{t}.
\end{align}
Considering $t\in\C$ such that $\mathrm{Im}(t)=0$, and then taking the limit as $t\rightarrow 0$ completes the proof. 
\end{proof}

\begin{remark}
From \eqref{eq:fz}, we can also conclude that $f$ is not complex G\^{a}teaux differentiable.
\end{remark}

In what follows we recall that $\S = \Spo$ and $\u_d = \widehat{\u}_d$.
\begin{corollary}
For $ \z$ and $\bm \xi $ in $ \HcurlS$, the $\R$-linear derivative of $j$ at $\z$ (given in \eqref{eq:reduced_cost}), in the direction $\bm \xi$  is given by
\begin{align}\label{def:dj}
 \mathrm{d}^{\mathbb{R}}  j(\z;\bm \xi)= 
 \mathrm{Re}\Big\{(\S \z-\u_d,\S\bm \xi)_{0,\Omega}
    + \alpha(\curlS,\curlS \bm \xi)_{0,\Gamma}
    + \beta(\z,\bm \xi)_{0,\Gamma}\Big\},
\end{align}
where $(\cdot,\cdot)_{0,\Gamma}$ denotes the inner product in both $L^2(\Gamma)$ and $\bm L^2(\Gamma)$. 
\end{corollary}
As usual we will avoid computing the term $\S \bm \xi$ by introducing the adjoint of $\S$, denoted by $\S^*$. 
Notice that since $\S$ is bounded linear, therefore  $\S^* : L^2(\Omega) \rightarrow \HcurlS^*$ is well-defined. 
Then from \eqref{def:dj}, given $\bm\xi \in \HcurlS$ we have that $\mathrm{d}^{\mathbb{R}}  j(\z;\bm \xi)$ is 
given by
\begin{align}
 \mathrm{d}^{\mathbb{R}}  j(\z;\bm \xi) = 
  \mathrm{Re}\Big\{\left\langle\S^*(\S\z-\u_d),\bm \xi\right\rangle_{\HcurlS^*,\HcurlS}
    +\alpha (\curlS \z,\curlS \bm \xi)_{0,\Gamma}
    + \beta(\z,\bm \xi)_{0,\Gamma}
    \Big\}. 
  \label{def:dj*}
\end{align}
We will introduce further assumptions on ${\bm \mu}$ and ${\bm \kappa}$ below so that this duality pairing in \eqref{def:dj*} becomes the inner product on $\Lt$. As usual, for optimal control of PDEs, $\S^*$ is the solution operator for a problem similar to the state equation called the {\it adjoint state equation}: for $\f \in \L$ 
find $\w\in \Hcurl$ such that 
\begin{align}\label{eq:Pre_adjoint1}
	\begin{aligned}
		\curl \left(  {\bm \mu}^{-1}\curl \w\right)-(\i\omega){\bm \kappa} \w &=  {\f} \quad \mbox{in } \Omega, \\
                   \w \times \n  &= \bm 0  \quad \mbox{on } \Gamma . 
	\end{aligned}                    
	\end{align}
	The weak formulation of this problem is: find  $\w\in \Hzcurl$ such that 
	\begin{align}
	 a^*(\w,\v)=(\f,\v) \qquad \forall \v \in \Hzcurl, \label{def:AdjointProb}
	\end{align}
where 
\begin{align}
  a^*(\w,\v):=  \overline{a(\v,\w)} \qquad \forall \w,\v \in \Hcurl. \label{def:AdjointForm}
\end{align}
By the Lax-Milgram lemma (see Lemma \ref{lemma:state_sol}), the above problem is well-posed.  Now,  given $\bm \xi$ in $\HcurlS$ we set $\u_{\bm \xi}=\S\bm \xi$, the unique solution to state equation, cf. (\ref{def:S}), namely   
	\begin{align}\label{eq:uxi}
	\begin{aligned}
	 a(\u_{\bm \xi}, \v)&=0 \qquad &&\forall\v \in \Hzcurl,\\
	 \gamma_t \u_{\bm \xi}&= \bm \xi\times \n \quad  &&\mbox{ on } \Gamma.
	\end{aligned} 
	\end{align}
In turn, given $\f\in \L$ let $\w$ be the solution for (\ref{def:AdjointProb}), and testing the first equation in \eqref{eq:Pre_adjoint1} with $\overline{\S\bm \xi}$, and integrating by parts  
 using \eqref{eqn:GreenCurl} and the fact that ${\bm \mu}^{-1}\curl \w \in \Hcurl$ we arrive at 
  \begin{align*}
    (\f,\S\bm \xi)_{0,\Omega}    
                       &= -\left\langle \gamma_t({\bm \mu}^{-1} \nabla \times  \w), \gamma_T \u_{\bm \xi}\right\rangle_{\Gamma^*}\\
                        &= -\left\langle \gamma_t({\bm \mu}^{-1} \nabla \times  \w),  \bm \xi \right\rangle_{\Gamma^*}.
 \end{align*}
Recall that $\langle\cdot,\cdot\rangle_{\Gamma^*}$ represents  the duality pairing  $H_{||}^{-\frac 12}(\divS;\Gamma)\!$ and $\!H_\perp^{-\frac 12}(\curlS;\Gamma)$. Nevertheless, under some additional regularity on $\bm \mu$ and $\bm \kappa$ we will show  that $\w$ is smooth enough so this duality becomes an integral. Now, by setting $\f=\u-\u_d$, \eqref{eq:Pre_adjoint1} becomes the \emph{adjoint problem} for the state equation, and we have the following result.
\begin{theorem}[]\label{thm:adjoint} 
 The adjoint operator for $\S$ is given by 
\begin{align}
 \S^*: \L&\to H_{||}^{-\frac12}(\divS;\Gamma) \nonumber\\ 
         \f&\mapsto \S^*\f =-\gamma_t({\bm \mu}^{-1} \nabla \times  \w),
\end{align}
 where $\w\in \Hzcurl$ solves 
    \begin{align}\label{eq:adjoint}
	\begin{aligned}
		\curl \left(  {\bm \mu}^{-1}\curl \w\right)-(\i\omega){\bm \kappa} \w &= \f \quad \mbox{in } \Omega \\
			\w \times \n &= 0 \quad \mbox{on } \Gamma. 
	\end{aligned}	                    
	\end{align}
Since, $\HcurlS  \hookrightarrow H_{\perp}^{-\frac12}(\curlS;\Gamma)$,  we can rewrite \eqref{def:dj} as 
\begin{align}  \label{def:dj1}
  \mathrm{d}^{\mathbb{R}}  j(\z;\bm \xi)= 
  \mathrm{Re}\Big\{\left\langle\S^*(\S\z-\u_d),\bm \xi\right\rangle_{\Gamma^*}
   +\alpha (\curlS \z,\curlS \bm \xi)_{0,\Gamma}
   +\beta (\z, \bm \xi)_{0,\Gamma}
  \Big\} , 
\end{align}
for all $\bm\xi \in \HcurlS$. 
\end{theorem}
\begin{proof}
The proof follows from the above analysis 
and the fact that ${\bm \mu}^{-1}\curl \w \in \Hcurl$, and therefore $\gamma_t({\bm \mu}^{-1} \nabla \times  \w) \in H_{||}^{-\frac12}(\divS;\Gamma)$, cf. Lemma \ref{lemma:surjTanTrace}. Finally, since $\HcurlS \hookrightarrow H_\perp^{-{\frac 12}}(\curlS;\Gamma)$, therefore, 
$H_{||}^{-\frac12}(\divS;\Gamma) = \left(H_\perp^{-{\frac 12}}(\curlS;\Gamma) \right)^* \hookrightarrow 
\HcurlS^*$. Thus, we can replace the $\HcurlS^*$--$\HcurlS$ duality pairing in \eqref{def:dj*} 
by the $\langle \cdot, \cdot \rangle_{\Gamma^*}$ pairing and the proof is complete. 
\end{proof}


\subsection{Additional regularity of $\S^*$ for Lipschitz polyhedra}
\label{s:addreg}

%
The goal of this section is to show that  we can obtain additional regularity for $\S^*$; for instance,  $\S^*: \L \to \Lt,$ under suitable assumptions on  ${\bm \mu}$ and ${\bm \kappa}$. 
These additional assumptions along with  the regularity of the adjoint problem (\ref{eq:adjoint}) will imply  $\gamma(  {\bm \mu}^{-1}\curl \w) \in \H^\sigma(\Gamma)$, for some $\sigma>0$. As a result, we can replace the duality pairing in \eqref{def:dj1} by the inner product  in $\bm L^2(\Gamma)$. 
From now on, we will assume that ${\bm \kappa}$ and ${\bm \mu}$ are in $ W^{1,\infty}(\Omega)$. In what follows, we give some technical results that are slight variations of the ones found in \cite[Thm. 4.1]{MR2957021}.  We omit most of the proof details as they mostly follow from straightforward calculations.

\begin{lemma} \label{lemma:lemma_1} 
 If ${\bm \kappa}\in W^{1,\infty}(\Omega)$  such that ${\bm \kappa}(x)\geq {\bm \kappa}_0>0$ almost everywhere, and $\phi\in \mathcal{C}_0^\infty (\Omega)$, then  ${\bm \kappa}^{-1}\phi \in H^1_0(\Omega)$ and 
\[
 \nabla \phi =\left({\bm \kappa}^{-1} \nabla {\bm \kappa}\right)\phi+ {\bm \kappa}\nabla \left({\bm \kappa}^{-1} \phi\right).
\]
\end{lemma}
\begin{proof}
	The proof follows immediately after using the product rule in $\nabla (\bm\kappa^{-1}\phi)$ and rearranging the resulting
	expression. 
\end{proof}

Using this result we obtain the following two lemmas involving the product rule for divergence and curl.
\begin{lemma} \label{lemma:div(ku)}
 Let ${\bm \kappa}\in W^{1,\infty}(\Omega)$  such that ${\bm \kappa}(x)\geq {\bm \kappa}_0>0$ almost everywhere,  and $\u\in \L$ such that ${\div {\bm \kappa}\u =v\in L^2(\Omega)}$, then $ \u\in \Hdiv$ and 
 \begin{align}
  \div(\u)={{\bm \kappa}^{-1}} v-\left( {{\bm \kappa}^{-1}}\nabla {\bm \kappa}\right)\cdot \u.
 \end{align}

\end{lemma}

\begin{lemma} \label{lemma:curl(ku)}
 Let $\zeta\in W^{1,\infty}(\Omega)$  such that $\zeta(x)\geq \zeta_0>0$ almost everywhere,  and $\u\in \L$ such that ${\curl \zeta\u =\v\in \L}$, then $ \u\in \Hcurl$ and 
 \begin{align}
  \curl(\u)=\zeta^{-1}\v+\u\times\left( {\zeta^{-1}}\nabla \zeta\right).
 \end{align}
\end{lemma}


We will  show that ${\bm \mu}^{-1}\nabla \times \w$ is smooth enough to have a well-defined  and integrable trace, where $\w$  is the solution for  the adjoint problem (\ref{eq:adjoint}).
To do that, We will now show some results for the regularity of $\w$, 

\begin{lemma}\label{lemma:gamma_n(curl)}
 If $\w\in \Hzcurl$, then $\curl\w \in \Hzdiv$.
\end{lemma} 
\begin{proof}
It is clear that $\div (\curl \w)=0,$ then $\curl \w\in \Hdiv$ and therefore $\gamma_{\n}(\curl \w)$  belongs to $ H^{-\frac 12}(\Gamma).$  Similarly; for each $v\in H^1(\Omega)$, $\nabla v\in \Hcurl $, and  
\begin{align*}
\langle\gamma_{\n}(\curl \w),\gamma v\rangle_{-\frac{1}{2},\frac{1}{2},\Gamma}=&
\int_{\Omega } v\, \div (\curl\w)d\x +\int_\Omega \curl\w\cdot \nabla vd\x\\
 =&~ 0+\int_\Omega \w\cdot \curl (\nabla v)d\x
      -\langle \gamma_t \w, \gamma_T(\nabla v)\rangle_{\Gamma^*}\\
 =&~0.
\end{align*}
Thus, from the surjectivity of the trace map $\gamma$ from $H^1(\Omega)$ onto $H^{\frac{1}{2}}(\Gamma)$ the proof concludes. 
\end{proof} 

%
\begin{theorem}\label{th:regularity1} Let $\w$ be the solution for the adjoint problem (\ref{eq:adjoint}) with  ${\bm \mu}, {\bm \kappa} \in W^{1,\infty}(\Omega)$, then  
\[
 {\bm \mu}^{-1}\curl \w \in \Hcurl\cap \Hzdiv.
\]
\end{theorem}
\begin{proof}
We will first invoke Lemma~\ref{lemma:div(ku)} with $\u = \G :={\bm \mu}^{-1}\curl\w$ and ${\bm \kappa} = {\bm \mu}$. 
Notice that $\G \in \L$ and $\div \bm\mu \G = \div \curl \w = 0 \in L^2(\Omega)$. Therefore, $\G \in \Hdiv$ and 
\begin{equation}
	\label{eq:bG}
	\div \G=-({\bm \mu}^{-1}\nabla {\bm \mu})\cdot ({\bm \mu}^{-1}\curl\w)\in L^2(\Omega). 
\end{equation}	
In addition, we notice that $\curl \G= (\u-\u_d)+\i\omega {\bm \kappa}\w\in \L$. Thus, $\G \in \Hcurl\cap \Hdiv$.  

Finally, we show that the normal trace of $\G$ vanishes to complete the proof. From Lemma~\ref{lemma:lemma_1}, $\bm \mu^{-1}v\in H^1(\Omega)$ for all $v\in H^1(\Omega)$, this along Lemma \ref{lemma:gamma_n(curl)}, and the Green's identity in $\Hdiv$, cf. \cite[Lemma 1.4]{Gatica14}, yields 
 \begin{align*}
 0= \left\langle \gamma_{\n}\left(\curl \w\right),\gamma(\bm \mu^{-1}v) \right\rangle_{-\frac 12,\frac 12,\Gamma}
 =&~ \int_{\Omega} \curl \w \cdot \nabla(\bm \mu^{-1}v)dx +\int_{\Omega} (\bm \mu^{-1} v)\div( \curl \w)dx\\
=&~ \int_{\Omega} \bm \mu^{-1} \curl \w \cdot  \bm \mu\nabla(\bm \mu^{-1}v)dx\\
=&~ \int_{\Omega} \bm \mu^{-1} \curl \w \cdot  \left(\nabla v- \bm \mu^{-1} \nabla \bm \mu v \right) dx . 
 \end{align*}
 Then using \eqref{eq:bG}, we obtain that
  \begin{align*}
 0= \left\langle \gamma_{\n}\left(\curl \w\right),\gamma(\bm \mu^{-1}v) \right\rangle_{-\frac 12,\frac 12,\Gamma}
=&~ \int_{\Omega} \bm \mu^{-1} \curl \w \cdot \nabla v dx+ \int_{\Omega} v\div(\bm \mu^{-1}\curl \w)dx\\
=&~ \left\langle \gamma_{\n}\left(\bm \mu^{-1} \curl \w\right),\gamma(v) \right\rangle_{-\frac 12,\frac 12,\Gamma} ,
 \end{align*}
 which concludes the proof. 
\end{proof}
\begin{corollary} \label{cor:intAdj} 
 Let $\w$ be the solution for the adjoint equation (\ref{eq:adjoint}), then 
 \begin{align*}
  \gamma_t({\bm \mu}^{-1}\nabla \times \w)\in \Lt. 
 \end{align*}
 \end{corollary}
\begin{proof}
 From \cite[Thm. 4.4]{MR1609607} $\Hcurl\cap \Hzdiv \hookrightarrow  \bm\H^{\frac12+\epsilon_*}(\Omega)$, for some $\epsilon_*>0$. Thus, according to  Theorem \ref{th:regularity1} we have $ {\bm \mu}^{-1}\nabla \times \w\in \bm\H^{\frac12+\epsilon_*}(\Omega)$, then by trace theorem,  $\gamma({\bm \mu}^{-1}\nabla \times \w)\in \H^{\epsilon_*}(\Gamma)$. In particular, $\gamma_t({\bm \mu}^{-1}\nabla \times \w)\in \Lt $.
\end{proof}

As a result of Corollary~\ref{cor:intAdj},  the duality paring in \eqref{def:dj1} becomes an integral, namely 
\begin{align*}
 \left\langle \gamma_t({\bm \mu}^{-1} \nabla \times  \w), \bm \xi \right\rangle_{\Gamma^*}=&~  \left\langle \gamma({\bm \mu}^{-1} \nabla \times  \w)\times \n, \bm \xi\right\rangle_{\Gamma^*}\\
 =&~\int_\Gamma ( \gamma({\bm \mu}^{-1} \nabla \times  \w)\times \n ) \cdot \overline{\bm \xi} \, \mathrm{d}S.
\end{align*}
\begin{corollary}
The function $\widehat{\z} \in Z$ is an optimal control for \eqref{eq:reduced_cost}, if and only if 
  \begin{align*}  
  &\widehat{\u}=\,\S\widehat{\z},\\
  &\widehat{\bm \zeta}=\,\S^*(\widehat{\u}-\u_d),\\
  &\mathrm{Re}\left\{\left(\widehat{\bm \zeta},\z-\widehat{\z}\right)_{0,\Gamma}\!+
  \alpha \left(\curlS \widehat{\z},\curlS (\z-\widehat{\z})\right)_{0,\Gamma}
  + \beta \left(\widehat\z,\z-\widehat{\z} \right)_{0,\Gamma}
  \right\}= 0, \quad \forall \z \in Z.
 \end{align*}
\end{corollary}
\begin{proof}
The result is a consequence of Theorems \ref{thm:optimality} and \ref{thm:adjoint}, and Corollary \ref{cor:intAdj}.
\end{proof}


 \section{Discrete Problem}
 \label{s:discprob}

Let  $\Omega$ be a connected polyhedral Lipschitz domain, and $\{\Th\}_h$ be a family of shape regular simplicial triangulation for $\Omega$. For a given mesh $\Th$  we denote the set of faces on the boundary by $\Gamma_h$, and the edges belonging to $\Gamma_h$   by $\EGh$. We consider a conforming finite element space for $\Hcurl$ denoted by $\Nkh$, we utilize the basis  given in \cite{Jay:2005}, here $k=0$ denotes the lowest order. Thus, given $\z\in \HcurlS$ and  $k \in \mathbb{N}\cup \{0 \}$, we consider the semi-discrete state equation: f{}ind $\u_h\in \Nkh$ such that 
	\begin{align}\label{eq:DiscreteState}
	\begin{aligned}
		\left(  {\bm \mu}^{-1}\curl \u_h, \curl \v_h \right)_{0,\Omega}+(\i\omega)
         \left({\bm \kappa} \u_h,\v_h\right)_{0,\Omega} &= 0 \quad \forall \v_h \in \mathcal{N}^h_{k0},  \\                   
                    \u_h \times \n &= \z\times \n		\quad \mbox{on }  \Gamma,
	\end{aligned}	                    
	\end{align}
	where   $ \mathcal{N}^h_{k0}:= \Nkh\cap \Hzcurl$. Because $\mathcal{N}^h_{k0}$ is a closed subspace of $\Hzcurl$ well-posedness of (\ref{eq:DiscreteState})  follows from the continuous case, except for how the boundary condition $\u_h \times \n = \z\times \n$ is imposed. This can be done in several ways; for instance, with a $L^2(\Gamma_h)$ projection plus taking an average for the edge dofs. 
	Nevertheless, for that approach approximation and commutativity properties seem rather dif{}ficult to prove. We impose the Dirichlet condition in (\ref{eq:DiscreteState})  via moments, which allows us to use the well-known approximation theory of N\'ed\'elec and Raviart-Thomas elements in 2D along the approximation theory developed in  \cite{MR3522965}.
	
\subsection{Imposition of discrete boundary conditions}
\label{s:discbc}

As we already mentioned, we impose the condition $\u_h \times \n = \z\times \n$ through a lifting defined by moments. Now, the natural strategy would be to use the 2D local moments for $H(\div\!;F)$, for each face  $F$ on $\Gamma_h$, which is a collection of linear integral equations; see for instance \cite[Sec. 3]{Gatica14}. Nevertheless, because we consider $\z\times \n$ instead of just $\z$, imposing  those moments turns out to be equivalent to imposing $\gamma_T \u_h = \z$ through the local moments for $H(\curl\!,F)$, cf. \cite[Sec. 5.5]{MR2059447}.  We now define a global lifting operator 
\begin{align}
\begin{aligned}
  \Lh:\mathcal{D}_{\Lh}\subseteq \HcurlS &\mapsto \Nkh\\
         \z &\mapsto \Lh\z:=\u_z^h,
\end{aligned} \label{def:lifting}        
\end{align}
where $\mathcal{D}_{\Lh}$ denotes the domain of $\Lh$ which is a finite subspace of $\HcurlS$ and its elements have well-defined and continuous moments on $\Gamma_h$, $\u^h_z$ is the unique element in $\Nkh$ which shares the  local 2D N\'ed\'elec moments with $\z$ for all the faces/edges in $\Gamma_h$, and it has zero interior moments. Because of the regularity needed for this lifting, we will choose the 
discrete control space  $Z^h$ so that the following holds 
\begin{align}
	Z^h\subseteq \mathcal{D}_{\Lh}\subseteq \gamma_T \mathcal{N}^h_j,~ \mbox{ for some } j\in \{0,...,k\}.
\end{align}
%


\subsection{ Discrete solution operator and cost functional}
\label{s:discsoln}

Let $j\in \{0,...,k  \}$ and define ${Z^h\!:= \bm R_h\!\cap\! Z}$, where 
	\begin{align}
	 \bm R_h &:= \left\{\gamma_T\u_h: \u_h\in  \mathcal{N}^h_j \right\}. 
	\end{align}
Now, let us introduce the discrete  optimization problem, 
\begin{align}
            \displaystyle \min_{\u_h,\z_h} \J(\u_h,\z_h) :=
            \frac{1}{2}\int_\Omega |\u_h-\u_d|^2\,d\x
            	+\frac{\alpha}{2}\int_{\Gamma} |\curlS \z_h|^2\,dS
            	+\frac{\beta}{2}\int_{\Gamma} |\z_h|^2\,dS , 
	\end{align}
	where $\u_h$ is the unique solution to (\ref{eq:DiscreteState}), for $\z=\z_h$. 
As in the continuous case, we reduce this problem with the help of a (discrete) solution operator     
\begin{align*}
\S_h:Z^h\subseteq \bm R_h & \mapsto \Nkh\\
                       \z_h & \mapsto \S_h \z_h =: \u_h. 
\end{align*}
To prove that $\S_h$ is bounded, with constant independent of $h$, let us consider the following lemma. 
\begin{lemma}
	 Let $\u$ and $\u_h$ be the solutions for the continuous and discrete state equations, (\ref{eq:state}) and (\ref{eq:DiscreteState}), respectively, with respective boundary data given by $\z\times \n$ and $\z_h\times \n$. Then, the following holds
	 \begin{align}
	  \|\u-\u_h\|_{\curl,\Omega}\leq C \left(\min_{\v_h\in \mathcal{N}_k^h}\|\u -\v_h\|_{\curl,\Omega}+\|\z-\z_h\|_{H^{-{\frac 12}}_{\perp}(\curlS;\Gamma)} \right), \label{eq:error_estimate} 
	 \end{align}
     where   $C$ only depends on the shape regularity of $\Th$, $\Omega$ and the eigenvalues of ${\bm \mu}$ and ${\bm \kappa}$.
	\end{lemma}
\begin{proof}
 From \cite[Corollary 4.4]{MR3522965}, we have 
 \begin{align*}
  \|\u-\u_h\|_{\curl,\Omega}\leq C \left(\min_{\v_h\in \mathcal{N}_k^h}\|\u -\v_h\|_{\curl,\Omega}+\|\z\times \n-\z_h\times \n\|_{H^{-{\frac 12}}_{||}(\divS;\Gamma)} \right)
 \end{align*}
 Finally, from the  well-known isometry for weak rotations  between $H^{-{\frac 12}}_{\perp}(\curlS;\Gamma)$ and $H^{-{\frac 12}}_{||}(\divS;\Gamma)$, see \cite[p. 448]{MR3930592}, the proof concludes.
\end{proof}
A useful result, related to (\ref{eq:error_estimate}), 
is the existence of a uniformly bounded operator
\begin{align*}
 L_h:  \bm R_h &\mapsto \Nkh \\
 \z_h&\mapsto \u_{\z_h},
\end{align*}
such that 
$
 \gamma_T(\u_{\z_h})= \z_h,$ and $\|\u_{\z_h}\|_{\curl,\Omega}\leq C \|\z_h\|_{\curlS},
$
where $C$ is independent of $h$, see \cite{MR3522965} and references within. Thus, $\S_h$ is linear and bounded independently of $h$. 
 We now introduce the reduced discrete cost functional 
\begin{align}
\begin{aligned}
 j_h: Z^h &\mapsto \R, \\
 \z_h&\mapsto j_h(\z_h):=\J(\S_h(\z_h),\z_h).
\end{aligned} \label{def:Discrete_reduced}
\end{align}
It is clear that,  $j_h$ is continuous and bounded independently of $h$. Now, under the same arguments as for the continuous case we have 
\begin{align}\label{def:dj_h}
  \mathrm{d}^{\mathbb{R}}  j_h(\z_h,\bm \xi_h)
  =& \mathrm{Re}\{\left(\S_h\z_h-\u_d,\S_h\bm \xi_h\right)_{0,\Omega}+\alpha (\curlS \z_h,\curlS \bm \xi_h)_{0,\Gamma}
  + \beta (\z_h, \bm \xi_h)_{0,\Gamma}\}\\\nonumber
  =& \mathrm{Re}\{\left\langle\S_h^*(\S_h\z_h-\u_d),\bm \xi_h\right\rangle_{\Gamma^*}+\alpha (\curlS \z_h,\curlS \bm \xi_h)_{0,\Gamma}
  + \beta (\z_h, \bm \xi_h)_{0,\Gamma}\}
\end{align}
for all $\z_h,$ and $\bm \xi_h $ in $ Z^h$. Where the action of the discrete adjoint operator $\S_h^*$ will de defined later, cf. (\ref{def:discreteAdjoint}).

\subsection{Particular choice of discrete  control space}
\label{s:ctrlspace}

%

For simplicity and because most of the theory for N\'ed\'elec elements focuses on this case, we consider the lowest order space for our discrete control $\z_h$, along with a higher order approximation for $\u_h$. As usual, we just construct the real-valued spaces. It turns out that, it is easier to inherit properties for $\z_h$  from the space of $\z_h\times \n $, which is related to the Raviart-Thomas space. 
In fact, given $F\in \Gamma_h$ we recall that the local Raviart-Thomas space  $\mathcal{R}T_k(F)$ is just a $\pi/{2}$-rotation of N\'ed\'elec's space  $\mathcal{N}_k(F)$, see \cite{Jay2:2005}. 
This idea can be easily generalized to the piecewise linear manifold $\Gamma_h$. In order to do that,  we f{}irst  introduce the Raviart-Thomas space for $\Gamma_h$    
\begin{align}
\mathcal{R}T_k(\Gamma_h):=\left\{\q\in \Lt\colon\, \forall F\in \Gamma_h,~ \q|_{F}\in \mathcal{R}T_k(F)~\mbox{ and }~ \forall \e\in \EGh,~ [\q\cdot\bm \nu]_{\e}=0 \right\}, 
\end{align}
where $ [\q\cdot\bm \nu]_{\e}=0$ denotes the continuity of $\q\cdot\bm \nu$ across the  edge $\e$, and  
 $\bm \nu$  is the 2D ``normal vector'' on $\e$, this  will be clarified later, see  (\ref{def:normalFace}). In the same manner, the space for  the discrete control $\z_h$  will be a subspace of  the N\'ed\'elec space for  $\Gamma_h$, 
which is given by
\begin{align}
\mathcal{N}_k(\Gamma_h):=\left\{\q\in \Lt\colon\, \forall F\in \Gamma_h,~ \q|_{F}\in \mathcal{N}_k(F)~\mbox{ and }~ \forall \e\in \EGh,~ [\q]_{\e}\cdot\t_{\e}=0 \right\},
\label{def:NedelecManifold}
\end{align}
where $\t_{\e}$ is a unitary vector along the edge $\e$. Then, it follows that $\mathcal{R}T_k(\Gamma_h)$ is a  $\pi/{2}$-rotation of $\mathcal{N}_k(\Gamma_h)$ but just facewise, which can be compactly represented as $\mathcal{R}T_k(\Gamma_h)=\mathcal{N}_k(\Gamma_h)\times \n$.  Now, we show that 
for the lowest order case,
$k=0$, the identity $\mathcal{R}T_k(\Gamma_h)=\mathcal{N}_k(\Gamma_h)\times \n$ can be reduced to an identity between elements of a particular choice of bases. To show that, notice that  $\ds \mathcal{R}T_0(\Gamma_h)$ is given by
\begin{align*}
\left\{\q\in \Lt\colon\, \forall F\in \Gamma_h,\exists \bm a\in \mathbb{R}^3,\exists b\in \R ,~ \q|_{F}= \bm a+b\x~\mbox{ and }~ \forall \e\in \EGh,~ [\q\cdot\bm \nu]_{\e}=0 \right\}. 
\end{align*}
For  a global basis for this space we consider the Rao-Wilton-Glisson basis \cite{Rao1982}. 
From now on, we assume that all the faces on $\Gamma_h$ are oriented counterclockwise in terms of its vertices, and the normal vector of a face points outward. Now, let us consider an edge $\e\in \EGh$ and the two faces, $F\!\!_{+}$ and $F\!\!_{-}$ on $\Gamma_h$, which share $\e$.  Therefore, without loss of generality, we assume that  $F\!\!_{-}$ has $\e$ with a negative orientation.   We now want to find    
 $\bm \psi_{\e}$ in $\mathcal{R}T_0(\Gamma_h) \cap  H^{\frac12}_{-}(\Gamma)$ such that  $\bm \psi_{\e}\cdot \bm \nu|_{F\!\!_{+}}+\bm \psi_{\e}\cdot \bm \nu|_{F\!\!_{-}}$ vanishes point-wise on $\Eh$.
 Here the main difference between this case and the 2D-case is that  $\bm \nu|_{F\!\!_{+}}\neq  -\bm \nu|_{F\!\!_{-}}$ in general. Nevertheless, ``the definition for Raviart-Thomas basis'' is the same for this case.  Consider,
\begin{align}\label{def:RTbasis}
\begin{aligned}
  \bm \psi_{\e}: \Gamma_h &\mapsto  \mathbb{R}^3\!\!,\\
 \bm x &\mapsto \bm \psi_{\e}(\x):=
 \begin{cases}
 ~~  \displaystyle\frac{|\e|}{2|F\!\!_{+}|} (\x-\v^{+}_{\e}) & :\x\in F\!\!_{+},\\[0.3cm]
  - \displaystyle\frac{|\e|}{2|F\!\!_{-}|} (\x-\v^{-}_{\e}) & :\x\in F\!\!_{-},\\[0.3cm]
  \quad \bm 0 & : \mbox{elsewhere},
 \end{cases}
\end{aligned}
\end{align}
where $\v^{\pm}_{\e}$ is the vertex opposite to $\e$ in $F_{\pm}$, see Figure \ref{fig:Psi_e}, $\bm \psi_{\e}$ was scaled for visualization purposes. 
 \begin{figure}[H]
  \centering
  \includegraphics[width=0.32\textwidth]{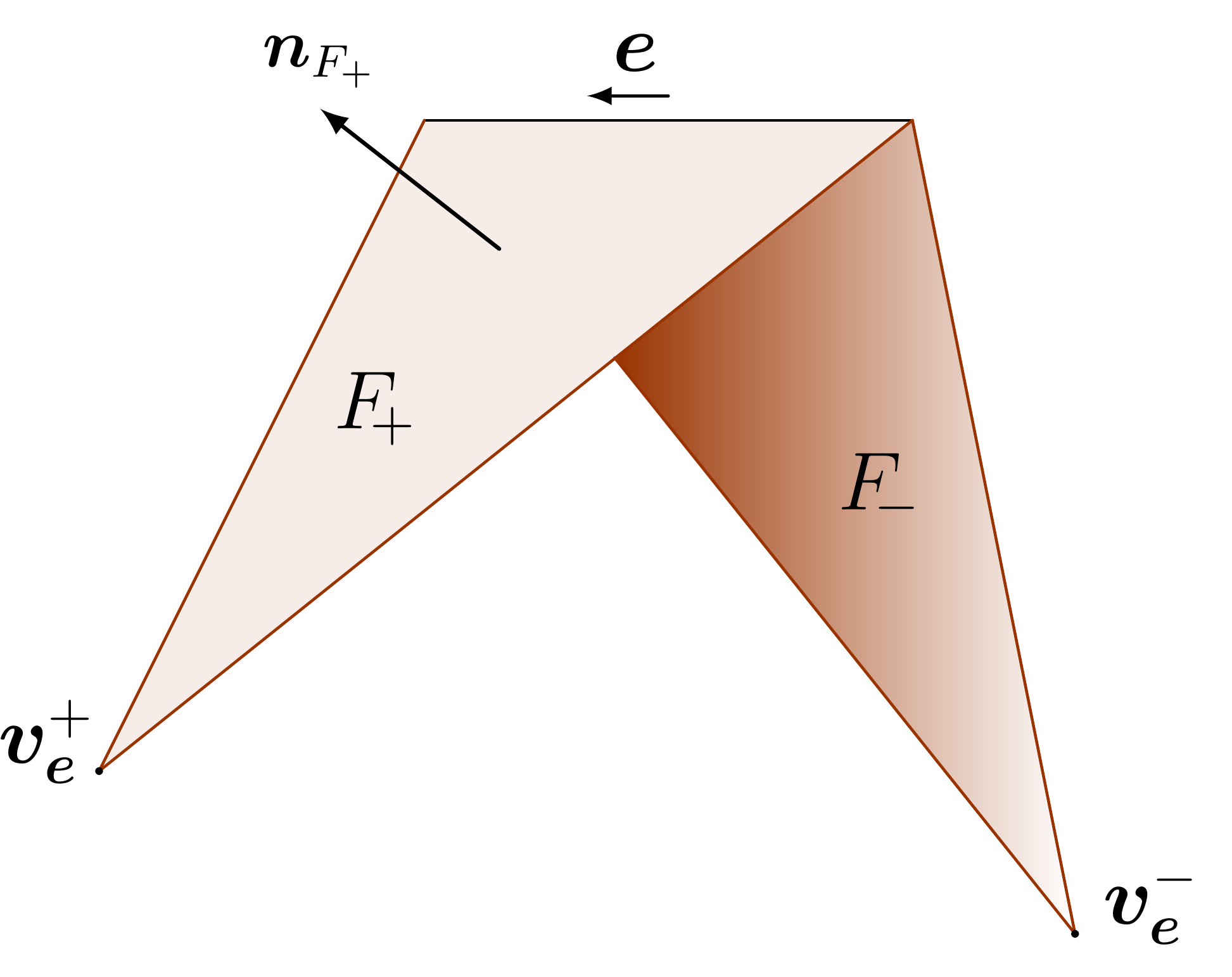}
  \includegraphics[width=0.32\textwidth]{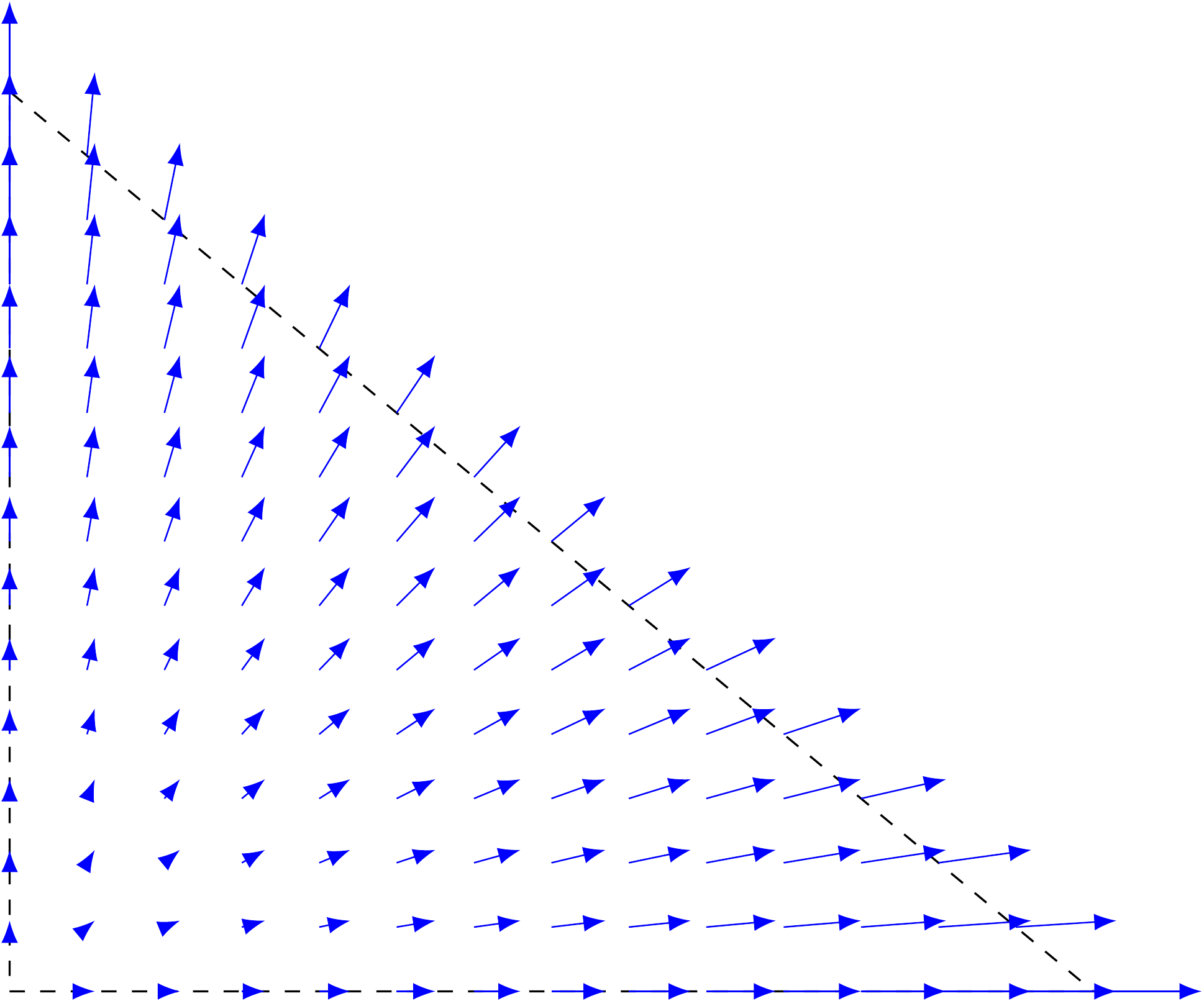}
   \includegraphics[width=0.32\textwidth]{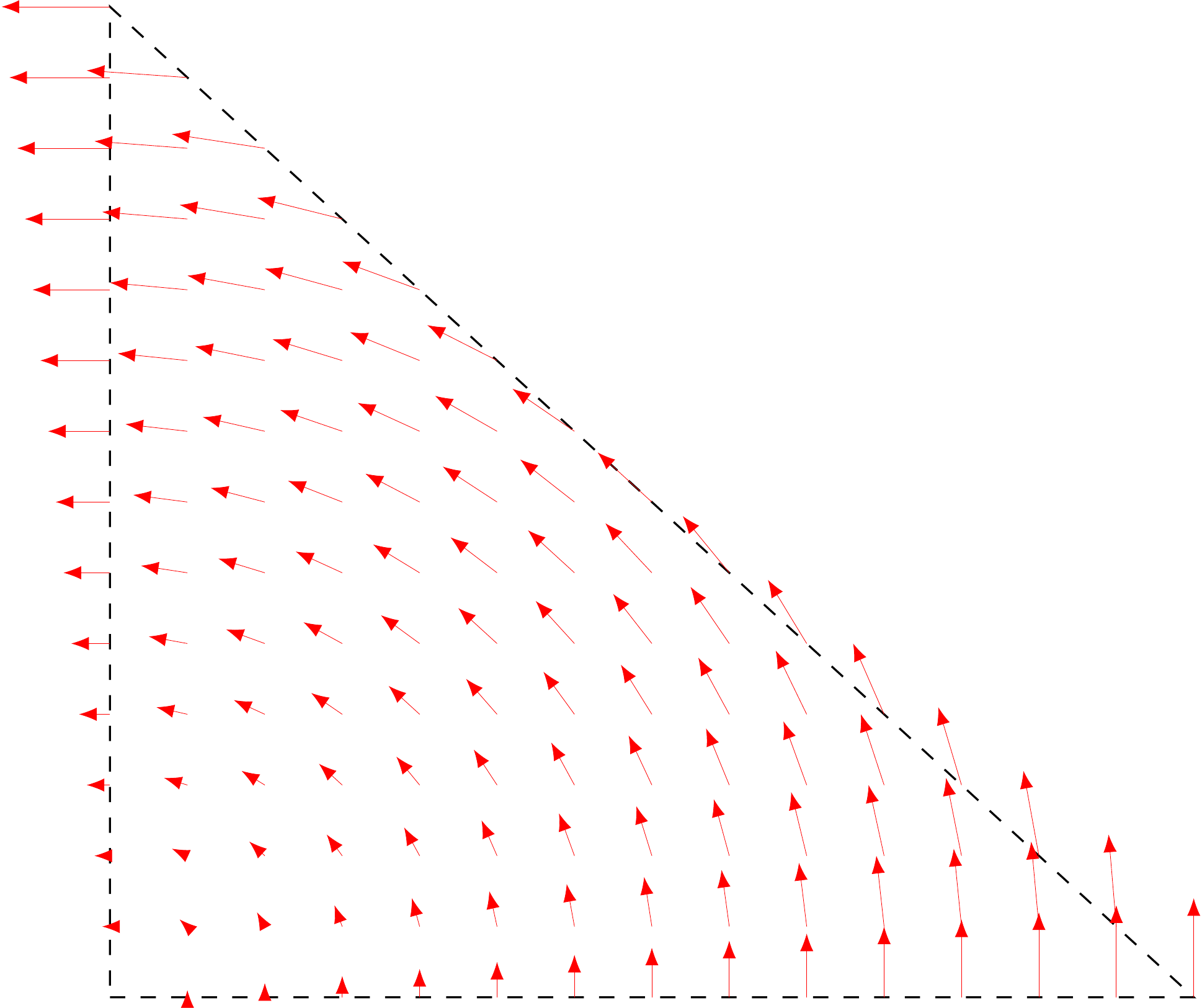}
  \caption{$F\!_+\cup \bm e\cup F\!\!_{-}$, $\bm \psi_{\e}|_{F\!\!_{+}}$ and  $(\n_{F\!\!_{+}} \times \bm \psi_{\e})|_{F\!\!_{+}}$ }
  \label{fig:Psi_e}
 \end{figure}
The functions $\bm \psi_{\e}$ have many good properties. For instance, it has a constant normal component on each edge of $\overline{F\!\!_{+}}$ and $ \overline{F\!\!_{-}}$. To show this, we follow  \cite[Sec. 4]{MR2194203}.  Given $F\in\Gamma_h$, let us assume that $F=\mathrm{conv}\{\v_1,\v_2,\v_3 \}$, and def{}ine $\e_j$ to be the edge opposite to the vertex $\v_j$. Then, by some basic properties of triangles and  orthogonal projections in  $\mathbb{R}^2$, see Figure \ref{fig:GeoTri},  yields
\begin{align}
 (\x-\v_j)\cdot \bm \nu_j= h_j=2\frac{|F|}{|\e_j|}\quad \forall \x\in \e_j, \quad \forall j\in \{1,2,3 \}.\label{eq:constNormals}
\end{align}
 \begin{figure}[H]
  \centering
  \includegraphics[width=2.2in,height=1.5in]{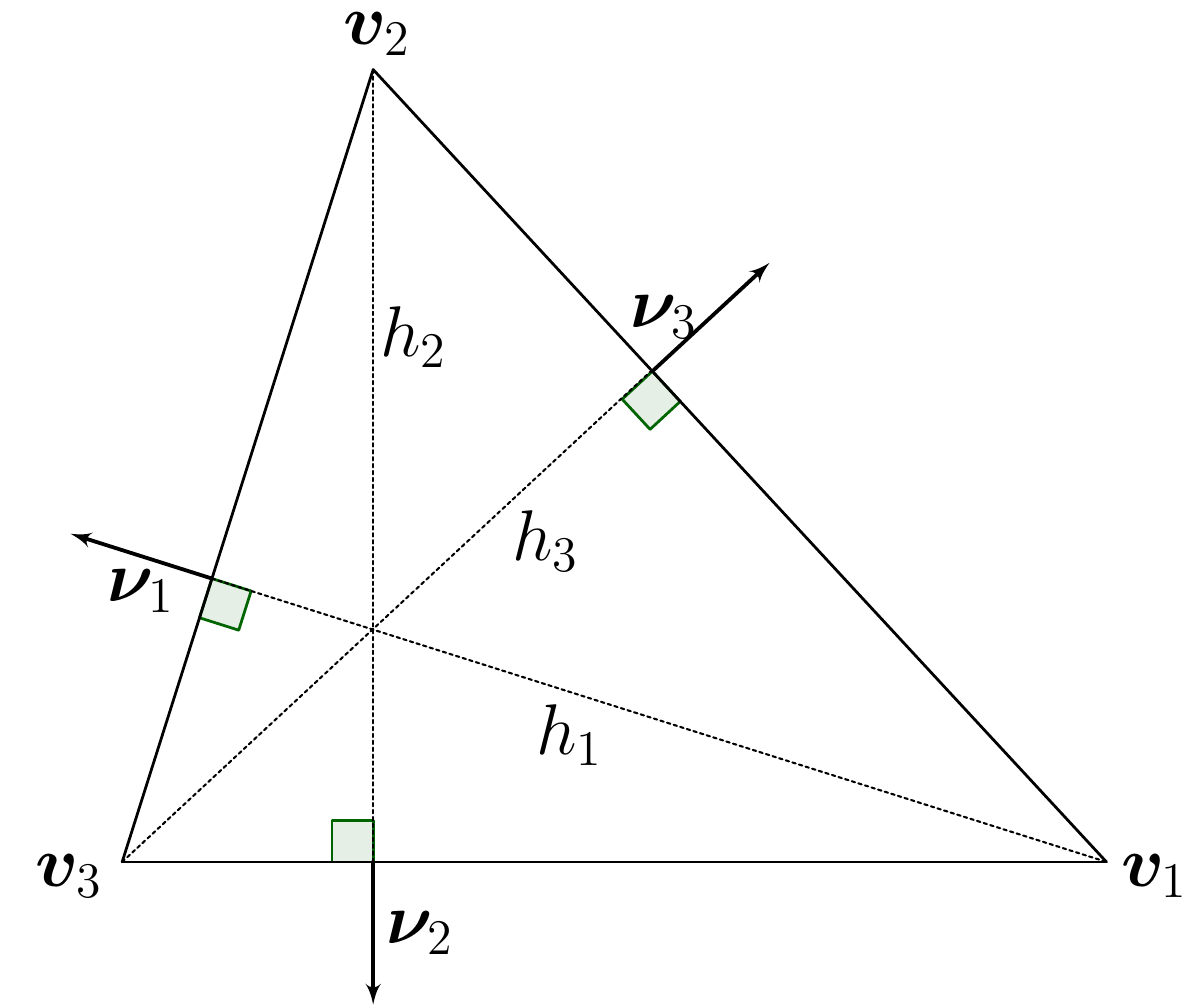}
  \caption{Altitudes and outer normals of $F$\!.}\label{fig:GeoTri}
 \end{figure}

Now, we give a formal definition for the ``normal'' vector $\bm \nu$, see  \cite[section 2]{buffa:I},
\begin{align}
 \bm \nu= 
  \begin{cases}
   \t_{\e}\times \n_{F\!\!_{+}} & \mbox{ on } \e \cap \overline{F}\!\!_+,  \\ \label{def:normalFace}
   \t_{\e}\times \n_{F\!\!_{-}} &  \mbox{ on } \e \cap \overline{F}\!\!_-,
  \end{cases}
\end{align}
where $\t_{\e}$ is a unitary vector along $\e$, following its orientation, and $ \n_{F_{\pm}}$ is the unitary outer normal vector to $F_{\pm}$. Note that, $ \bm {\nu}|_{\e\cap F\!\!_\pm}=\pm \bm \nu_{F\!\!_\pm}$, therefore $\bm \nu|_{\e\cap F\!\!_{-}}$ points inward. With this definition we have the following result

\begin{lemma}[] Let $\bm\psi_e$ def{}ined as in (\ref{def:RTbasis}), then 
 \begin{align*}
   \left(\bm \psi_e\cdot \bm \nu\right)(\x) =&
   \begin{cases}
    \ds 1 &: \x \in \e, \\
     0 &: \x \in  \EGh\setminus\{\e\},
   \end{cases}\\
  \divS \bm \psi_e  = \,&
  \begin{cases}
     \pm  \frac{\ds |\e|}{\ds|F_{\pm}|} & \mbox{ in }  F_{\pm},\\
     0 & \mbox{elsewhere},
  \end{cases}\\
  \int_\Gamma   \divS \bm \psi_e dS=\,&0.
 \end{align*}
\end{lemma}
\begin{proof}
 See \cite[Lemma 4.1]{MR2194203}.
\end{proof}
Now, we will study a basis for 
$\N_0(\Gamma_h)$, cf. (\ref{def:NedelecManifold}). To simplify the notation, from now on we assume  that $\e$ is the edge between the vertices $[\x_\ell,\x_m]$, following that orientation. 
Because for the lowest order case, in  2D and 3D, there are only edge related functions of the form, 
$ \widetilde{\bm \phi}_{\e} = \lambda_\ell\nabla \lambda_m-\lambda_m\nabla \lambda_\ell.$
This motivates us to consider the collection of functions $\{ \bm \phi_{\e}\}_{\e\in \EGh}$, given by  
\begin{align}\label{def:Nbasis}
\begin{aligned}
  \bm \phi_{\e}: \Gamma_h &\mapsto  \mathbb{R}^3\!\!,\\
 \bm x &\mapsto \bm \phi_{\e}(\x):=
 \begin{cases}
  |\e|\left(\lambda_\ell\nabla_\Gamma \lambda_{m}-\lambda_m\nabla_\Gamma \lambda_\ell\right)|_{F_{\pm}}& :\x\in F\!\!_{\pm},\\[0.3cm]
  \quad \bm 0 & : \mbox{elsewhere}.
 \end{cases}
\end{aligned}
\end{align}
Note that $\bm \phi_{\e}=\n_{F\!\!_{+}} \times \bm \psi_{\e}$, see Figure \ref{fig:Psi_e}.  Based on this observation, we have the following result  

\begin{lemma}\label{lemma:surfaceBasis} Let $\e\in \EGh$ then
\begin{align*}
\bm \phi_{\e}\cdot \t_{\widehat{\e}}= & 
\begin{cases}
 \ds 1 & \e=\,\widehat{\e},\\
 0 & \e\neq \widehat{\e},
\end{cases}\\
\bm \phi_{\e}\times  \n=&\,  \bm \psi_{\e},\\ 
 \curlS \bm \phi_{\e}  =  \,&
  \begin{cases}
     \pm  \frac{\ds |\e|}{\ds|F_{\pm}|} & \mbox{ in }  F_{\pm},\\
     0 & \mbox{elsewhere},
  \end{cases}\\
 \int_{\Gamma}\curlS   \bm \phi_{\e}dS=&\,0.
\end{align*}
\end{lemma}
\begin{proof}
 It follows from $\nabla_\Gamma \lambda_\ell\cdot \t_{\e}=\nabla  \lambda_\ell\cdot \t_{\e},$ the definition of $\bm \nu$, cf. (\ref{def:normalFace}), $\bm \phi_{\e}\times  \n$ having the same edge moments as $ \bm \psi_{\e}$, the unisolvence of $\mathcal{R}T_0(F)$, and the identity 
 $\curlS \bm\phi_{\e} =\divS \bm \phi_{\e}\times \n$. 
\end{proof}
We now show how to compute the terms that appears in the  stabilization term in the cost functional, and its derivative, for the discrete setting. Our analysis only involves the length of the boundary edges, their local orientation, and 
the barycentric coordinates $\widehat \lambda_1,\widehat \lambda_2$ and $\widehat \lambda_3$  on the 2D reference face $\widehat F=\mathrm{conv}\{(0,0)^T,(1,0)^T,(0,1)^T\}$.
\begin{proposition}
Let   $F\in \Gamma_h$ with edges $(\e_a,\e_b,\e_c)$, $\z\in Z^h$ with  
$\z|_F=\beta_a \bm \phi_{\e_a}+\beta_b \bm \phi_{\e_b}+\beta_c \bm \phi_{\e_c}$. Then, for each $i,j\in \{a,b,c\}$, yields 
\begin{align}
 \langle \beta_i \curlS\bm \phi_{\e_i},\beta_j \curlS\bm \phi_{\e_j} \rangle_F=&~\beta_i \overline{\beta}_j \int_F \curlS   \bm \phi_{\e_i} \curlS   \bm \phi_{\e_j}dS= \beta_i \overline{\beta}_j\frac{|\e_i||\e_j|}{|F|},\\
\langle \beta_i \bm \phi_{\e_i},\beta_j \bm \phi_{\e_j} \rangle_F=&~
\beta_i \overline{\beta}_j \int_F \bm \phi_i\cdot \bm \phi_jdS
=\beta_i \overline{\beta}_j |\bm e_i||\bm e_j| \frac{|F|}{2}\int_{\widehat F} \left(\widehat{\bm \phi}_\ell\right)^t\mathbb{B}_F \widehat{\bm \phi}_m dS,
\end{align}
where $\ell, m\in \{1,2,3 \}$,  
\begin{align*}
\mathbb{B}_F:=& \frac{1}{4|F|^2}
 \begin{bmatrix}
    |\e_a|^2 & \ds \frac{|\e_b|^2-(|\e_a|^2+|\e_c|^2)}{2} \\
    \ds \frac{|\e_b|^2-(|\e_a|^2+|\e_c|^2)}{2} & |\e_c|^2
 \end{bmatrix},\\
 \widehat{\bm \phi}_1:=&  -(\widehat\lambda_1+\widehat\lambda_2) 
 \begin{pmatrix}1\\0\end{pmatrix}
- \widehat\lambda_2\begin{pmatrix}0\\1\end{pmatrix}\!, 
\widehat{\bm \phi}_2:=
\widehat\lambda_3\begin{pmatrix}1\\0\end{pmatrix}- \widehat\lambda_2\begin{pmatrix}0\\1\end{pmatrix}\!, 
\mbox{ and }
\widehat{\bm \phi}_3:=
\widehat\lambda_3\begin{pmatrix}1\\0\end{pmatrix}+ (\widehat\lambda_1+\widehat\lambda_3)\begin{pmatrix}0\\1\end{pmatrix}.
\end{align*}
\end{proposition}
\begin{proof} First of all, because we are dealing with tangent fields it is enough to show the 2D case. 
Let us consider a face $F=\mathrm{conv}\{\v_1,\v_2,\v_3 \}$, where $\v_j=(x_j,y_j)^T$ for $j\in\{1,2,3 \}$, and the reference face $\widehat F=\mathrm{conv}\{(0,0)^T,(1,0)^T,(0,1)^T\}$ with associated barycentric coordinates $\widehat \lambda_1, \widehat \lambda_2, \widehat \lambda_3$, and the affine map
\begin{align*}
 \eta_F:\widehat F &\mapsto F,\\
 \begin{pmatrix}
  \widehat x \\  \widehat  y
 \end{pmatrix}
 &\mapsto 
  \begin{pmatrix}
  x_2-x_1 & x_3-x_1 \\  y_2-y_1 & y_3-y_1
 \end{pmatrix}
 \begin{pmatrix}
  \widehat x \\  \widehat  y
 \end{pmatrix}
 +
  \begin{pmatrix}
  x_1 \\ y_1
 \end{pmatrix}.
\end{align*}

\begin{figure}[h!]
 \centering
 \includegraphics[scale=0.7]{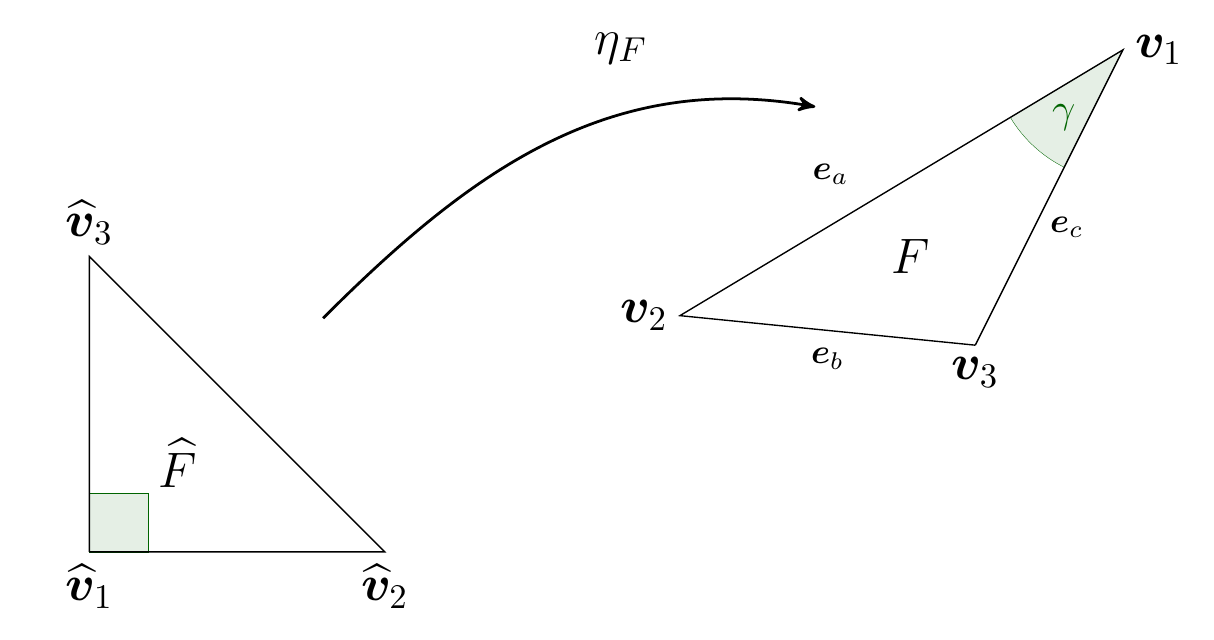}
 \caption{Af{}f{}ine transformation between $\widehat F$ and $F$}
 \label{fig:AffineMap}
\end{figure}
Let us assume the edges of $F$ have length $|\e_a|,|\e_b|$ and $|\e_c|$, see Figure \ref{fig:AffineMap}. And, as usual, define  
 \begin{align}
    B_F=
    \begin{pmatrix}
     x_2-x_1 & x_3-x_1\\
     y_2-y_1 & y_3-y_1
    \end{pmatrix}.
 \end{align}
 It straightforward to show 
 \begin{align*}
  det(B_F)=& |F|/|\widehat F|=2|F|,\\
      B_F^{-1}=& \frac{1}{2|F|}
    \begin{pmatrix}
     y_3-y_1 & x_1-x_3\\
     y_1-y_2 & x_2-x_1
    \end{pmatrix},\\
     B_F^{-1}B_F^{-T}=& \frac{1}{4|F|^2}
    \begin{pmatrix}
     |\bm e_a|^2                        & \frac{|\e_b|^2-(|\bm e_a|^2+|\bm e_c|^2)}{2}\\
     \frac{|\e_b|^2-(|\bm e_a|^2+|\bm e_c|^2)}{2} & |\bm e_c|^2
    \end{pmatrix},
 \end{align*}
where we have used $\overrightarrow{\v_1 \v_2}\cdot \overrightarrow{\v_1 \v_3}=|\e_a||\e_c|\cos(\gamma)$, see Figure \ref{fig:AffineMap}, and 
the law of cosines
\begin{align*}
 |\e_b|^2=& |\e_a|^2+|\e_c|^2-2|\e_a||\e_c|\cos(\gamma).
\end{align*}
In practice we replace $|\e_j|$ by $\pm|\e_j|$, according to the orientation of the edge on $F$. The rest of the proof follows from the fact that the elements of our basis are up to a constant, the length of the edge, are the same as the standard lowest order N\'ed\'elec basis of the first kind. 
\end{proof}

A different approach to compute the above mentioned quantities can be done in terms of $\bm \phi_\e\times \n$, for the 2D case, this can be found in \cite[sec. 4.2-4.3]{MR2194203}.

\subsection{Discrete Adjoint equation and optimality conditions}
\label{s:discopt}

Because for $\w_h\in \Nkh$, \newline ${{\bm \mu}^{-1}\nabla \times \w_h\notin \Hcurl}$, we cannot apply the same strategy as for the continuous adjoint $\S^*\!\!,$ cf. Theorem \ref{thm:adjoint}. To overcome this problem, we  f{}irst consider a discretization of the continuous adjoint problem, cf. (\ref{eq:adjoint}): f{}ind $\w_h \in \mathcal{N}^h_{k0}$ such that, 
	\begin{align}\label{eq:DiscreteAdjoint}
	\begin{aligned}
		a^*(\w_h,\v_h)
         &= (\S_h\z_h-\u_d,\v_h)_{0,\Omega} \quad \forall \v_h \in \mathcal{N}^h_{k0},  \\
	\end{aligned}	                    
	\end{align}
and as in the continuous case, for $\bm \xi_h$ and $\z_h$ in $Z_h$ we want to simplify the term 
\begin{align*}
 (\S_h\z_h-\u_d,\S_h\bm \xi_h)_{0,\Omega}
\end{align*}
which appears in the definition of $d^\R j_h$, so it does not require to compute/assemble the term $\S_h\bm \xi_h$, for each feasible direction $\bm \xi_h$. To do that, we apply the following splitting 
\begin{align*}
 \S_h\bm \xi_h= \S_{h0}\bm \xi_h + \Lh\bm \xi_h,
\end{align*}
where $\S_{h0}\bm \xi_h  \in \mathcal{N}^h_{k0} $, and $\Lh$ is the lifting operator defined in (\ref{def:lifting}). Thus, we define 
\begin{align}
\begin{aligned}
\left\langle\S_{h}^*(\S\z-\u_d),\bm \xi\right\rangle_{\Gamma^*}&:=
(\S_h\z-\u_d,\S_h\bm \xi)_{0,\Omega}\\
  &=  (\S_h\z-\u_d,\S_{h0}\bm \xi)_{0,\Omega}+ (\S_h\z-\u_d,\Lh\bm \xi)_{0,\Omega}\\
  &\stackrel{(\ref{eq:DiscreteAdjoint})}{=}  
   a^*( \w_h, \S_{h0}\bm \xi)+ (\S_h\z-\u_d,\Lh\bm \xi)_{0,\Omega}\\
   &\stackrel{(\ref{def:AdjointForm})}{=}  
   \overline{a( \S_{h0}\bm \xi, \w_h)}+ (\S_h\z-\u_d,\Lh\bm \xi)_{0,\Omega}\\
  &=
  -\overline{a(\Lh\bm \xi, \w_h)}+ (\S_h\z-\u_d,\Lh\bm \xi)_{0,\Omega}, 
\end{aligned}\label{def:discreteAdjoint}
\end{align}
where the last identity follows from (\ref{eq:DiscreteState}), taking $\v_h=\w_h.$ 
Note that  the support of $\Lh\bm \xi_h$ is contained just in a small $h$-neighborhood of $\Gamma_h$.
\begin{theorem}
 The f{}irst-order optimality condition (\ref{eqn:1stOrder}) implies in the discrete setting that:   $\barz_h\in Z^h$ is an optimal control  
 of (\ref{def:Discrete_reduced}) if and only if
  \begin{align*}
  &\baru_h=\,\S_h\barz_h,\\
  &\mathrm{Re}\left\{
  \left\langle\S_{h}^*(\baru_h-\u_d),\z_h-\barz_h\right\rangle_{\Gamma^*}+
  \alpha \left(\curlS \barz_h,\curlS (\z_h-\barz_h)\right)_{0,\Gamma}
  + \beta \left(\barz_h,\z_h-\barz \right)_{0,\Gamma}
  \right\}= 0, \quad \forall \z_h \in Z^h.
 \end{align*}
\end{theorem}

\subsection{Convergence of fully discrete scheme}
\label{s:convan}

\begin{theorem}\label{thm:ctrlconv}
 Let $\{\barz_h \}_h$ be the family of discrete optimal controls related to $\{\Th \}_h$ then
 \begin{enumerate}
  \item $\{\barz_h\}_h$ is uniformly bounded, 
  \item there exists a subsequence  $\left\{\barz_\tth \right\}_\tth$ of $ \{\barz_h \}_h$ such that
  \begin{align*}
   \|\z_\tth-\barz\|_{\curlS} \rightarrow 0, \quad \mbox{as } \tth \rightarrow 0 , 
  \end{align*}
  where $\barz$ is the unique solution to the continuous optimization problem (\ref{eq:reduced_cost}). 
 \end{enumerate}
\end{theorem}
\begin{proof}
 The first part of the proof  follows from 
 \begin{align}
  j_h(\barz_h)&\leq j_h(\bm 0)\Rightarrow { \frac{\min\{\alpha,\beta\}}{2} } \| \barz_h \|^2_{\curlS}\leq \frac{\alpha}{2}\| \curlS\barz_h\|^2_{0,\Gamma}+\frac{\beta}{2}\| \barz_h\|^2_{0,\Gamma} \leq  \|\u_d\|^2_{0,\Omega}.
 \end{align}
 Since  $\HcurlS$ is a  Hilbert space, $ \{\barz_h \}_h$ has  a weakly convergent subsequence i.e., there  exists 
 $\ttz \in \HcurlS$, such that
 \begin{align}
  \barz_\tth\rightharpoonup \ttz, \quad \mbox{ as } \tth\rightarrow 0.
 \end{align}
Now, since the injection of $\HcurlS$ into  $H^{-{\frac 12}}_{\perp}(\curlS;\Gamma)$ is compact, therefore 
\begin{align}
 \|\barz_\tth-\ttz\|_{H^{-{\frac 12}}_{\perp}(\curlS;\Gamma)}\xrightarrow{\tth} 0,~  \mbox{ when } \overline{\z}_h\rightharpoonup \ttz \mbox{ in } \HcurlS.
\end{align}
In turn, from (\ref{eq:error_estimate}) we get $\|\S_\tth\barz_\tth-\S\ttz\|_{\curl,\Omega}\xrightarrow{\tth}0$. In fact, 
$$\|\S_\tth\barz_\tth-\S\ttz\|_{\curl,\Omega}\lesssim \mathrm{dist}(\S\ttz,\Nkh)+\|\barz_\tth-\ttz\|_{H^{-{\frac 12}}_{\perp}(\curlS;\Gamma)}\xrightarrow{\tth}0.$$
Now, we will show convergence of the solutions  for the discrete adjoint problem, cf. (\ref{eq:DiscreteAdjoint}). In order to do that, let us consider $\widetilde{\w}\in \Hzcurl$ to be the unique solution of
\begin{align*}
 a^*(\widetilde{\w}, \v)=(\S\ttz-\u_d,\v) \quad \forall \v\in \Hzcurl,
\end{align*}
and   functionals $\ell$ and $\ell_h$ given by
\begin{align*}
 \ell(\v)=&~ (\S\ttz-\u_d,\v)_{0,\Omega} \quad  \forall\v\in \Hzcurl,\\
  \ell_h(\v_h)=&~ (\S_h\barz_h-\u_d,\v_h)_{0,\Omega} \quad  \forall \v_h\in \mathcal{N}^h_{k0}.
\end{align*}
Thus, from  Strang's f{}irst lemma we obtain that   
\begin{align*}
 \|\widetilde{\w}-\barw_\tth \|_{\curl,\Omega}\lesssim&
 \inf_{\v_\tth \in \mathcal{N}^\tth_{k0}}  \|\widetilde{\w}-\v_\tth \|_{\curl,\Omega}
 + \sup_{\v_\tth \in \mathcal{N}^\tth_{k0}} \left |\frac{\ell(\v_\tth)-\ell_h(\v_\tth)}{\|\v_\tth\|_{\curl,\Omega}} \right|\\
 \lesssim& 
  \inf_{\v_\tth \in \mathcal{N}^\tth_{k0}}  \|\widetilde{\w}-\v_\tth \|_{\curl,\Omega}+ \|\S\ttz - \S_\tth\barz_\tth\|_{\curl,\Omega}\xrightarrow{\tth}0.
\end{align*}
On the other hand, it is clear that for each $\bm p\in Z$ we have that 
\begin{align*}
 \alpha \left(\curlS \barz_h,\curlS \bm p\right)_{0,\Gamma}
  + \beta \left(\barz_h,\bm p \right)_{0,\Gamma}\xrightarrow{\tth} 
  \alpha \left(\curlS \ttz,\curlS \bm p\right)_{0,\Gamma}
  + \beta \left(\ttz,\bm p \right)_{0,\Gamma}.
\end{align*}
In turn,  $\{Z^h\}_h$ is dense in $\HcurlS$; moreover, $\{Z^h\}_h$ is dense in $H^{-{\frac 12}}_{\perp}(\curlS;\Gamma)$ which follows from  \cite[Corollary 5]{Buffa2003}. Thus, for a given $\bm p\in Z$ we define $\bm p_h$ as its best approximation  in $Z^h$, then 
\begin{align*}
 \left\langle\S_{\tth}^*(\S_\tth\barz_\tth-\u_d),\bm p_\tth\right\rangle_{\Gamma^*}
             &=(\S_\tth\barz_\tth-\u_d,\S_\tth\bm p_\tth)_{0,\Omega}\\
             &=(\S_\tth\barz_\tth-\u_d,\S_\tth\bm p_\tth-\S \bm p)_{0,\Omega}+(\S_\tth\barz_\tth-\u_d,\S\bm p)_{0,\Omega}.
\end{align*}
Thus, because $\S_h\barz_h$ is bounded (independently of $h$) we conclude   
\begin{align*}
 \left\langle\S_{\tth}^*(\S_\tth\barz_\tth-\u_d),\bm p_\tth\right\rangle_{\Gamma^*} \xrightarrow{\tth}  
  \left\langle\S^*(\S\ttz-\u_d),\bm p\right\rangle_{\Gamma^*}.
\end{align*}
Finally, 
\begin{align*}
  \mathrm{Re}\left\{ \left\langle\S^*(\S\ttz-\u_d),\bm p\right\rangle_{\Gamma^*}+ 
  \alpha \left(\curlS \ttz,\curlS \bm p\right)_{0,\Gamma}
  + \beta \left(\ttz,\bm p \right)_{0,\Gamma}
  \right\}=0 \quad \forall \bm p\in Z,
\end{align*}
and by uniqueness of local minimum we conclude $(\ttz,\S\ttz,\widetilde{\w})=(\barz,\baru,\barw).$
\end{proof}

\subsection{Using $\| \curlS {\bf z}\|^2_{0,\Gamma}$ as the regularization term}
\label{s:reg}
In this section we study the problem 
\begin{align}\label{eq:reduced_cost_2}
            \displaystyle  \min_{\z\in Z^2} \widetilde{j}(\z) =  
            \ds  \min_{\z\in Z} \frac{1}{2}\int_\Omega |\Spo \z-\widehat{\u}_d|^2\,d\x+\frac{\alpha}{2}\int_{\Gamma} |\curlS\z|^2\,dS,
\end{align}
where 
\begin{align}
\label{eq:Zad2}
  Z^2:=&  \bigl\{\z\in \HcurlS:    \curlS \z\in L_0^2(\Gamma) \bigr\}\cap  \mathrm{ker}(\curlS)^\perp.
\end{align}
The continuous analysis for this problem follows from the previous case, given that 
\begin{align}
\|\z\|_{0,\Gamma}^2 + \|\curlS\z\|_{0,\Gamma}^2 \lesssim  \|\curlS\z\|_{0,\Gamma}^2 \qquad \forall \z \in Z^2,
\end{align}
which follows from the open mapping theorem \cite[Corollary~2.7]{HBrezis_2011a}, and the  following result
\begin{lemma} \label{lemma:decoL2t} The operators 
\begin{align*}
 \divS:\Lt\mapsto H^{-1}(\Gamma)/\mathbb{R}, \mbox{ and }~ \curlS:\Lt\mapsto H^{-1}(\Gamma)/\mathbb{R}
\end{align*}
 are  linear, continuous and surjective. Moreover,
 \begin{align*}
  \mathrm{ker}(\divS)=&~\{ {\bf curl_\Gamma}\varphi:\, \varphi\in H^1(\Gamma) \},\\
  \mathrm{ker}(\curlS)=&~\{ \nabla_\Gamma\varphi:\, \varphi\in H^1(\Gamma) \}.
 \end{align*}
\end{lemma}
\begin{proof}
 See Def{}inition 2.3, Remark 3.2 and Proposition 3.1 in \cite{buffa:II}.
\end{proof}
In turn, for the discrete case we have 
\begin{proposition} The set of discrete admissible controls, $Z^h\!,$ satisfies $Z^h\subset Z^2.$
\end{proposition}
\begin{proof}
 From Lemma \ref{lemma:surfaceBasis} it is clear that 
 \begin{align*}
  Z^h\subset \bigl\{\z\in \HcurlS:    \curlS \z\in L_0^2(\Gamma) \bigr\}.
 \end{align*}
On the other hand, we have the following Hodge decomposition for $\Lt$, cf. \cite[Thm. 3.4]{buffa:II},
\begin{align*}
 \Lt=\nabla_\Gamma H^1(\Gamma)\bm\oplus \bm \curlS H^1(\Gamma).
\end{align*}
Thanks to Lemma \ref{lemma:decoL2t}, to conclude the proof  we only need to show
\begin{align*}
 \divS Z^h=\{0 \}.
\end{align*}
In order to do that, let $\bm r(\bm x)$ be the position vector associated with the point $\bm x$, which satisfies $\curl \bm r\equiv \theta$. Then, given  an edge $\e\in \Eh$ and faces $F\!\!_{+}$ and $F\!\!_{-}$ in $\Gamma_h$ such that 
$\e=\overline{F}\!_{-}\cap \overline{F}\!_{+}$, cf. Figure \ref{fig:Psi_e}, we have 
\begin{align*}
 \divS \bm \phi_{\e}|_{F\!_{\pm}}= \curlS \bm \psi_{\e}|_{F\!_{\pm}}
 = \pm \frac{|\e|}{|F\!_{\pm}|}\curlS\left(  (\bm r-\bm r(\v^{\pm}_{\e}))\right)
 =\pm \frac{|\e|}{|F\!_{\pm}|} (\curl (\bm r-\bm r(\v^{\pm}_{\e}) ))\cdot \n_F=0,
\end{align*}
which concludes the proof.
\end{proof}

\section{Numerical results}
\label{s:num}

In this section we test our codes, developed in  \textsc{Matlab}$^\circledR$, for dealing with higher order N\'ed\'elec spaces with nonhomegeneous Dirichlet boundary conditions. For optimization, we use complex version of the BFGS algorithm \cite{Sorber12}.

\subsection{Code validation for N\'ed\'elec elements }
To test our code we consider the problem involving electrodes given in  \cite{BeRoSa2005-2,BeRoSa2005} but with a simpler Dirichlet boundary condition. Let us  
denote  $\bm H, \bm E$ and $\bm J$ be  the complex amplitude of magnetic field, electric field, and the current  density on a bounded domain $\Omega$, F{}igure \ref{fig:electrode}. And let us consider the  eddy current problem: find $(\bm H, \bm E,\bm J)$ such that 

\begin{minipage}[t]{0.3\textwidth}
  \begin{center}
  \includegraphics[width=2in,height=2.1in]{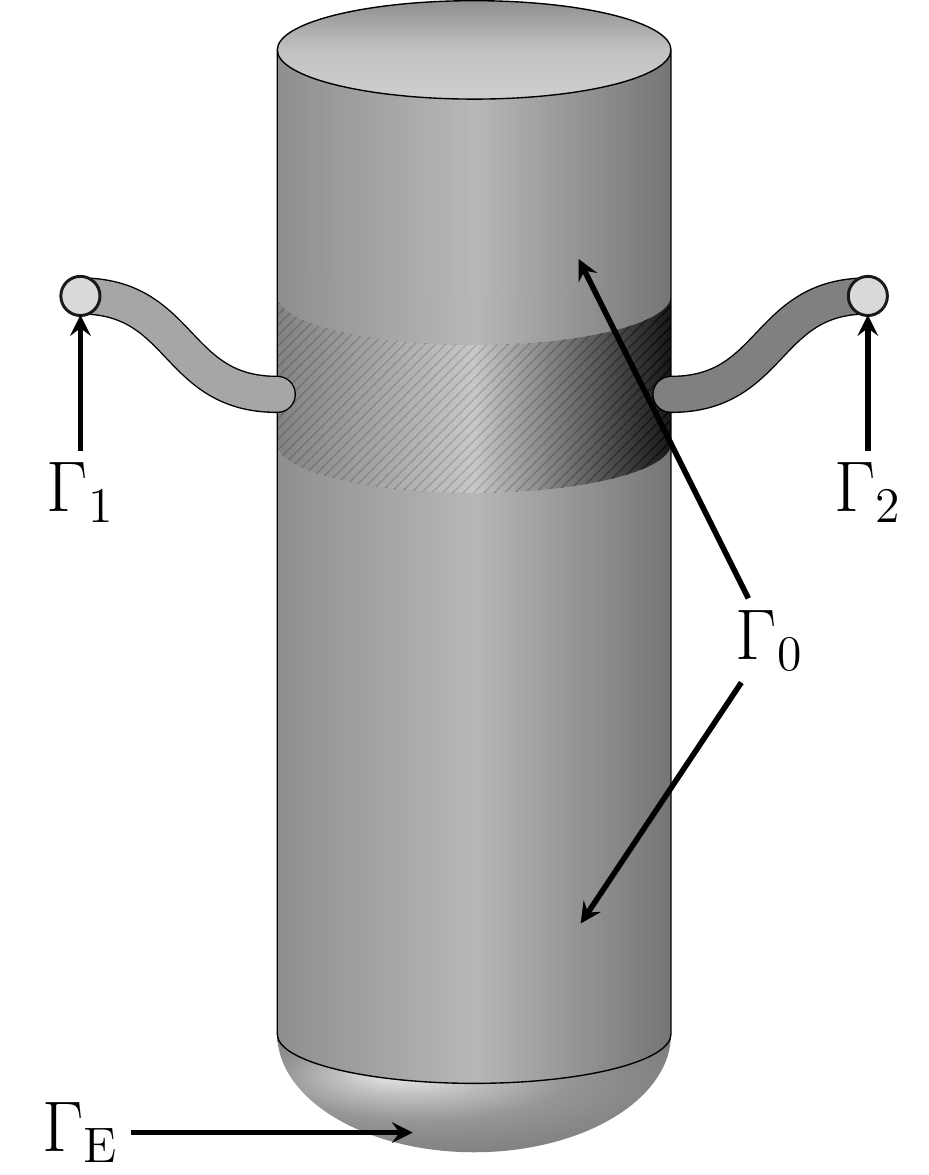}
 \captionof{figure}{Electrode}
  \label{fig:electrode}
  \end{center}
\end{minipage}\hspace*{3cm}
\begin{minipage}[t]{0.5\textwidth}
\begin{align}\label{eq:Rodolfo}
 \curl \bm H =&~ \bm J \quad   \mbox{ in } \Omega,\nonumber\\
 i\omega \mu \bm H + \curl \bm E =&~ \bm 0 \quad  \mbox{ in } \Omega,\nonumber\\
 \bm J =&~ \sigma \bm E  \quad  \mbox{ in } \Omega,\nonumber\\
\intertext{subject to:} 
  \bm E \times \n  =&~\bm 0 \qquad  \mbox{ on }  \Gamma_E,\\
  \int_{\Gamma_n} \bm J\cdot \n =&~ \iota_n, \qquad n\in \{1,..., N \},\nonumber\\
  \bm E \times \n =&~ \bm 0 \qquad \mbox{ on }  \bigcup_{k=1}^N\Gamma_k,\nonumber\\
  \bm J\cdot \n =&~ 0 \qquad \mbox{ on } \Gamma_0,\nonumber\\
  \bm H\cdot \n =&~ 0 \qquad \mbox{ on } \partial \Omega.\nonumber\\\nonumber
 \end{align}
\end{minipage}

In the particular case that, $\Omega$ is a cylinder of radius $R$ and  height $L$, $N=1$,   and the partition for the boundary $\Gamma$ is given by 
 \begin{align*}
  \Gamma_1=\{(x,y,z): z=L, ~x^2+y^2<R  \}, ~ \Gamma_{ E}=\{(x,y,z): z=0, ~x^2+y^2<R  \},
  \Gamma_0 = \Gamma \setminus \overline{(\Gamma_1\cup \Gamma_E)}. 
 \end{align*}
 Then it is possible to f{}ind an analytic solution for  (\ref{eq:Rodolfo}), cf.  \cite{BeRoSa2005-2,BeRoSa2005}, given by
\begin{align}\label{def:HEJ}
 \bm H(x,y,z) =&\frac{\iota_1}{2\pi R} \frac{I_1(\gamma r)}{I_1(\gamma R)}\bm e_\theta,~
 \bm E(x,y,z) = \frac{\iota_1\gamma}{2\pi R\sigma} \frac{I_0(\gamma r)}{I_1(\gamma R)}\bm e_z,~
 \mbox{ and } \bm J=\sigma \bm E,
\end{align}
where $I_\nu$ is the modif{}ied  Bessel functions of the f{}irst kind of order $\nu$, $\gamma :=\sqrt{i\omega \mu \sigma}$, 
$
 r:= \sqrt{x^2+y^2}$, $ 
 \bm e_\theta := r^{-1}\left(-y,x,0\right), \mbox{ and }
 \bm e_z:=(0,0,1).
$
Now, we approximate $\bm H$ with  $\mathcal{N}^h_{2}$ which satisfies \newline
${\mathcal{P}^3_2\varsubsetneq \mathcal{N}^h_{2} \varsubsetneq \mathcal{P}^3_3}$. Also, we set all the parameters equal to $1$ except for $R=\frac 12$, and then   we consider the problem: f{}ind $\bm H_h\in \mathcal{N}^h_{2}$ such that
\begin{align*}
 a(\bm H_h,\v_h)=&~\bm 0 \quad \forall \v_h \in \mathcal{N}^h_{2}\cap \Hzcurl,\\
 \bm H_h\times \n =&~ \bm H\times \n ~ \mbox{ on } \Gamma_h.
\end{align*}
The domain $\Omega$ was approximated with a family of meshes generated with Gmsh \cite{Gmsh}. Figure  \ref{fig:ErrorRodolfo}
shows a log-log plot for the error $\mathcal{E}_h:= \|\bm H-\bm H_h\|_{\curl,\Omega}$, along with the coarsest mesh considered, where the color represents the element-wise error $ \|\bm H-\bm H_h\|_{\curl,K}$.
\begin{figure}[H] 
 \centering
 \includegraphics[scale=0.5]{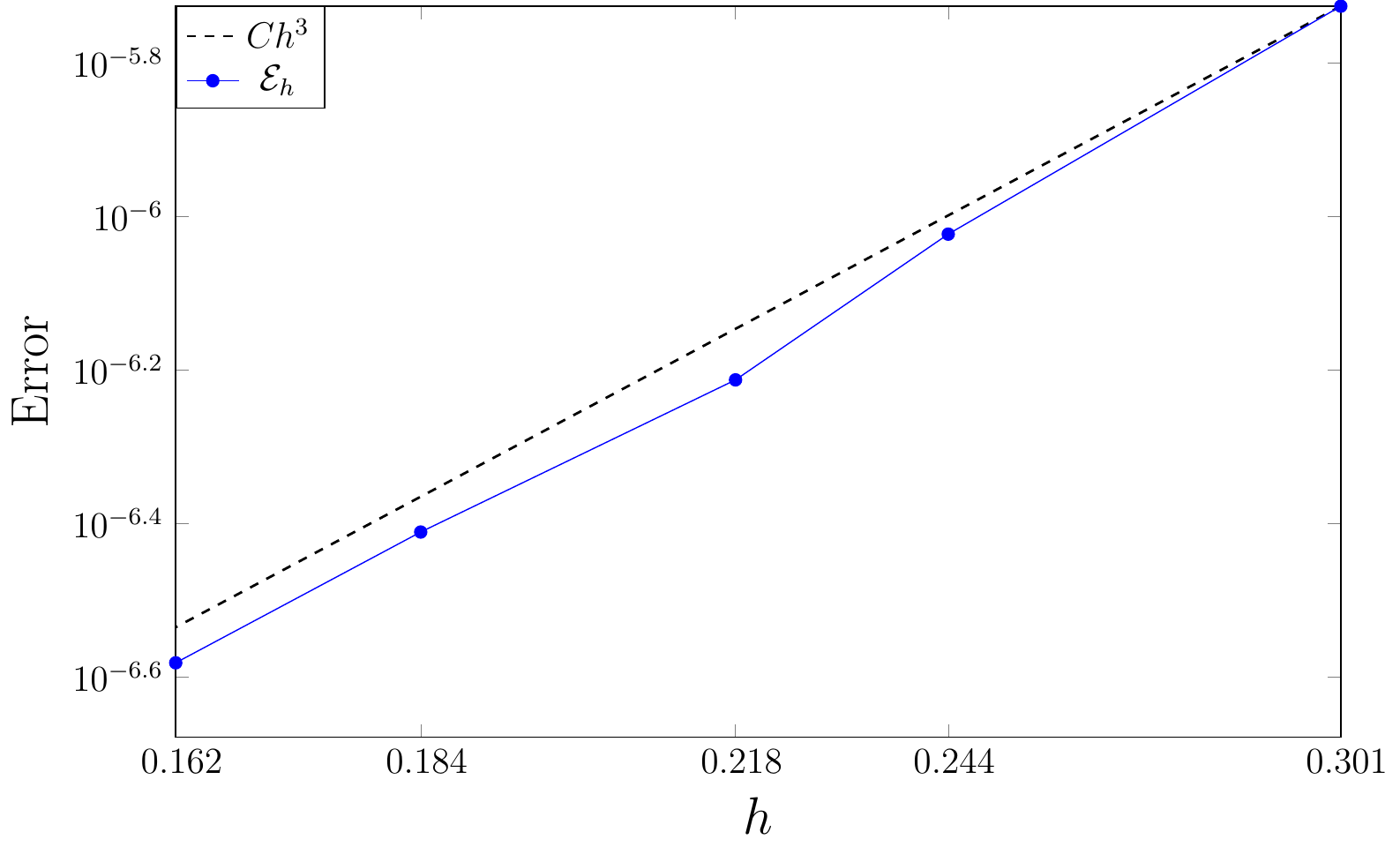}
  \includegraphics[scale=0.5]{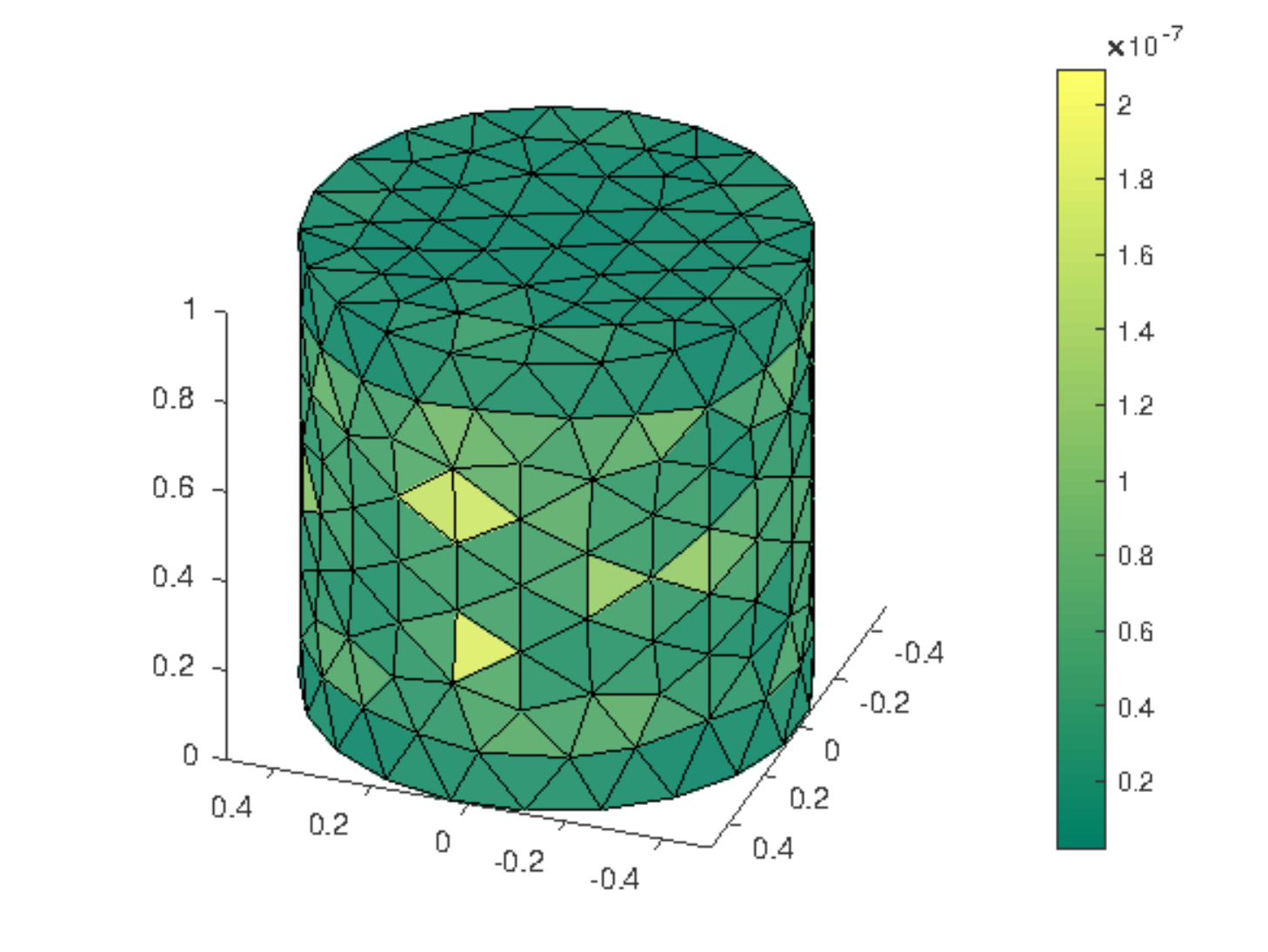}
 \caption{Log-log plot $\mathcal{E}_h$ vs $h$, and coarsest mesh along with element-wise error.}
 \label{fig:ErrorRodolfo}
\end{figure}

\subsection{Validation optimization routines} 
We devote this section to test our codes related to the minimization problem. We start by testing our 
equivalent expression for $d^\R \!j_h$, cf. (\ref{def:dj_h}). The main difficulty is the term that involves
$\S_h^*$, which can be computed as in (\ref{def:discreteAdjoint}). 
Figure~\ref{f:fd} (left) shows the the difference between $d^\R\!j_h(\z;\bm \xi)$ and its approximation using
a finite difference quotient. Here $\bm\xi$ denotes the random direction. As expected, we observe a linear
rate of convergence until the round-off error kicks in. 
\begin{figure}[!h]
 \centering
 \includegraphics[width=0.59\textwidth]{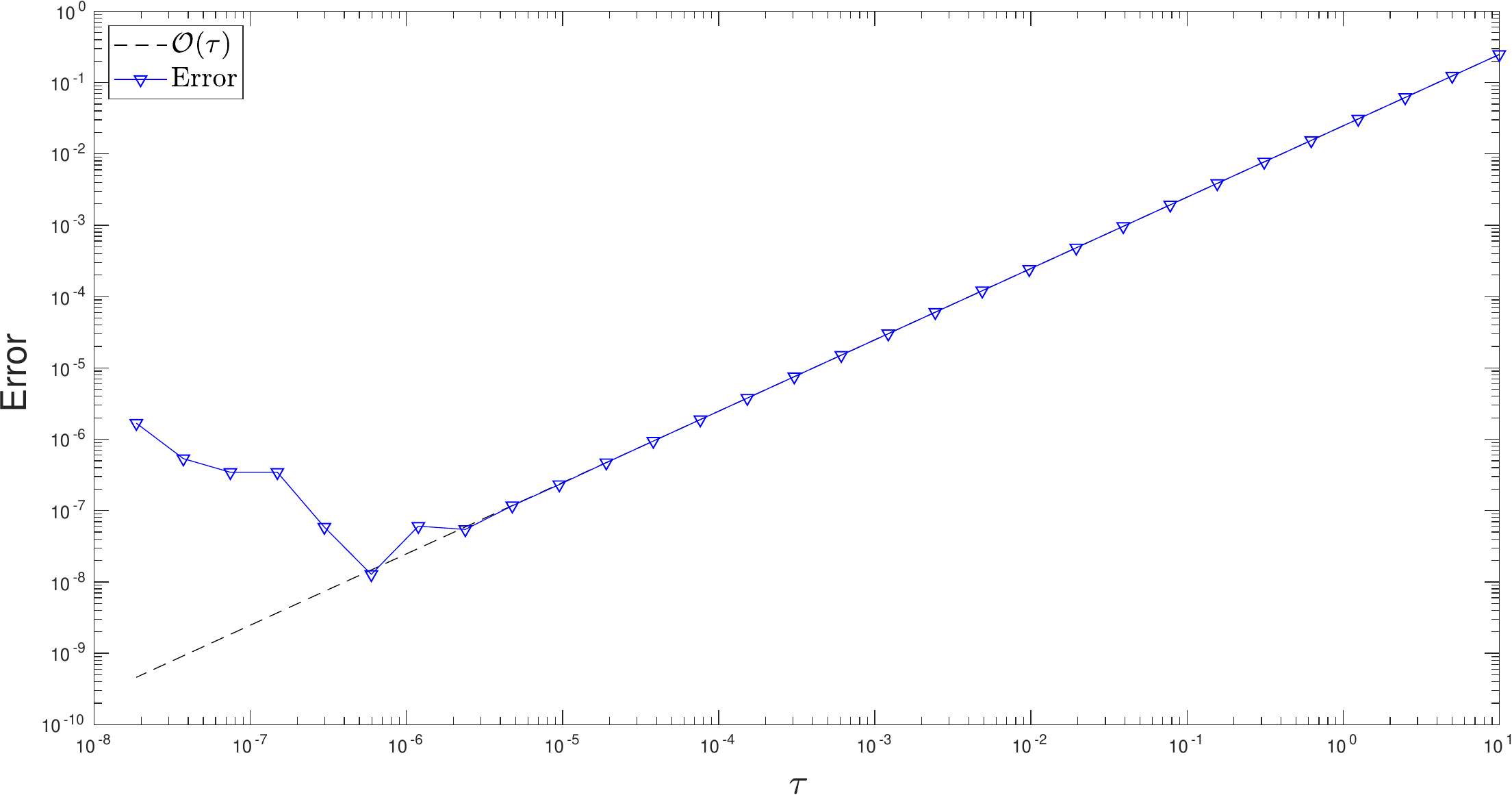}
 \includegraphics[width=0.4\textwidth]{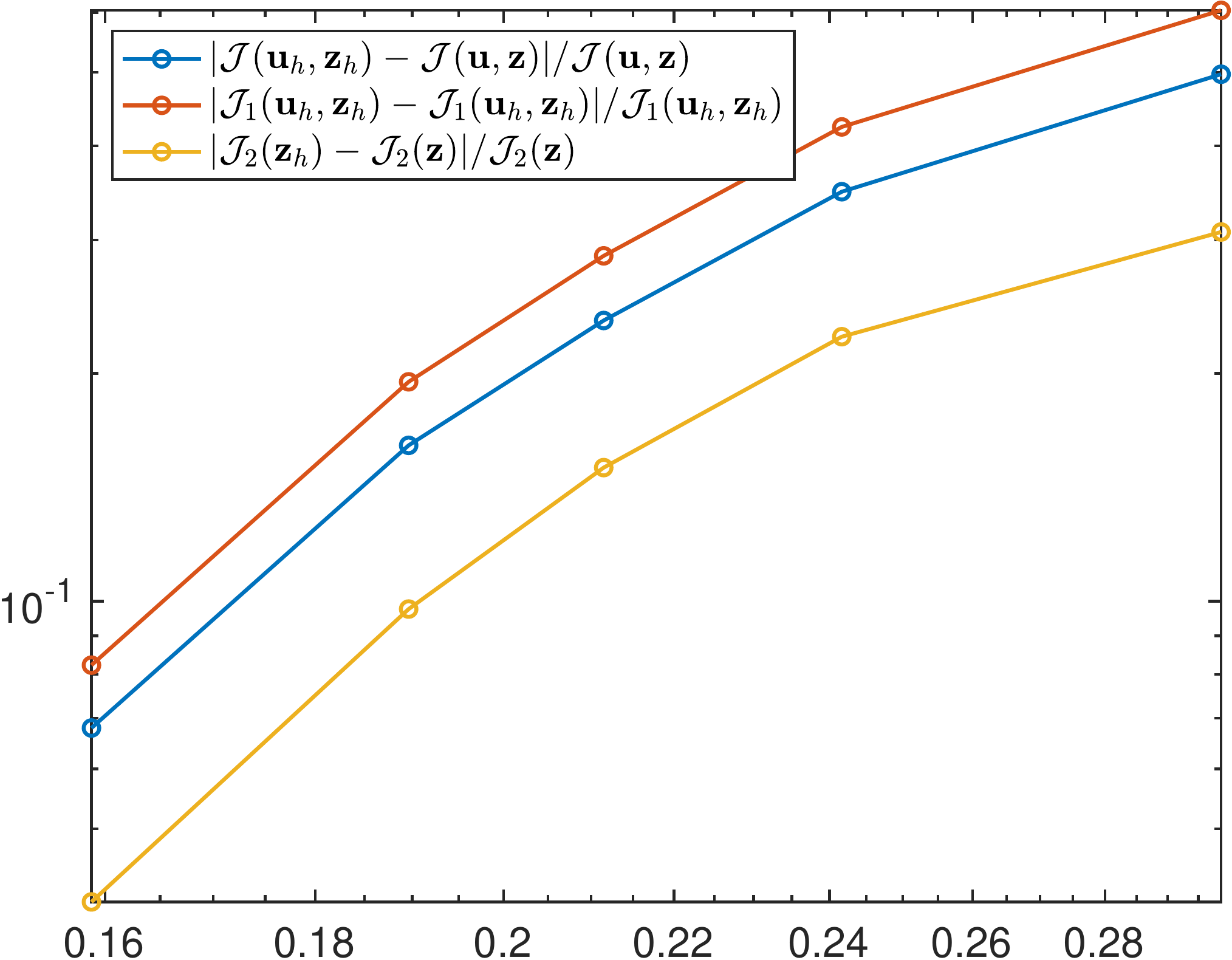}
 \caption{\label{f:fd} 
 {\bf Left:} Given a random direction $\bm\xi$, the panel shows the difference between
 $d^\R\!j_h(\z;\bm \xi)$ and its finite difference approximation. As expected, we observe a linear 
 rate of convergence. 
 {\bf Right:} We let $\alpha=1e^{-3}$ and $\beta = 0$ in the cost functional $\mathcal{J}(\cdot)$. Let 
 $\mathcal{J}_1(\u,\z) := \frac12 \|\u-\u_d\|^2$ and $\mathcal{J}_2(\z) := \frac{\alpha}{2} \| \curlS \z\|_{L^2(\Gamma)}^2$. 
 Moreover, let $\z$ be the optimal control corresponding to the finest mesh. Then the three curves illustrate 
 $|\mathcal{J}(\u_h,\z_h)-\mathcal{J}(\u,\z)|/\mathcal{J}(\u,\z)$, 
$|\mathcal{J}_1(\u_h,\z_h)-\mathcal{J}_1(\u,\z)|/\mathcal{J}_1(\u,\z)$, 
and $|\mathcal{J}_2(\z_h)-\mathcal{J}_2(\z)|/\mathcal{J}_2(\z)$ as $h \rightarrow 0$.
 }
\end{figure}

\subsection{Convergence of optimization problem}
It is challenging to construct an exact solution to the optimal control problem to directly show 
the applicability of Theorem~\ref{thm:ctrlconv}. Instead, we show the convergence of the cost 
functional $\mathcal{J}(\u_h,\z_h)$ to $\mathcal{J}(\u,\z)$ as $h \rightarrow 0$. Here $(\u,\z)$ is 
the optimal control corresponding to 
a mesh obtained after 6 refinements. The optimization problem is solved using the BFGS method
mentioned above with a stopping tolerance of $10^{-9}$. We let $\u_d = \bm H $, cf. (\ref{def:HEJ}) 
and let $\alpha = 10^{-3}$ and $\beta = 0$. 
Let $\mathcal{J}_1(\u,\z) := \frac12 \|\u-\u_d\|^2$ and $\mathcal{J}_2(\z) := \frac{\alpha}{2} \| \curlS \z\|_{L^2(\Gamma)}^2$. 
Figure~\ref{f:fd} (right) shows the errors $|\mathcal{J}(\u_h,\z_h)-\mathcal{J}(\u,\z)|/\mathcal{J}(\u,\z)$, 
$|\mathcal{J}_1(\u_h,\z_h)-\mathcal{J}_1(\u,\z)|/\mathcal{J}_1(\u,\z)$, 
and $|\mathcal{J}_2(\z_h)-\mathcal{J}_2(\z)|/\mathcal{J}_2(\z)$ as $h \rightarrow 0$. The expected convergence is observed.
\section*{Acknowledgement}
The authors would like to thank Thomas S. Brown and Peter Monk for reading this manuscript
and helpful comments. The authors would also like to thank Irwin Yousept and Pablo Venegas for several comments 
and pointing us to multiple useful references.

\bibliographystyle{plain}
\bibliography{references}

\end{document}